\documentclass[11pt,twoside]{article}
\usepackage{bbm}
\usepackage{mathrsfs}
\usepackage{amssymb}
\usepackage{amsmath}
\usepackage{amsthm}
\usepackage{amsfonts}
\usepackage{color}
\usepackage{booktabs}
\usepackage{graphicx}
\usepackage[active]{srcltx}
\usepackage{CJK}
\usepackage{geometry}
\usepackage[cp1252]{inputenc}

\usepackage{mathrsfs}
\usepackage{graphicx}

\usepackage[active]{srcltx}

\usepackage{authblk}

\allowdisplaybreaks

\usepackage{titletoc}
\titlecontents{section}[0pt]{\addvspace{2pt}\filright}
              {\contentspush{\thecontentslabel\ }}
              {}{\titlerule*[8pt]{.}\contentspage}

\def\rr{{\mathbb R}}
\def\rn{{{\rr}^n}}

\def\nn{{\mathbb N}}

\def\bs{{\mathbb S}}

\def\fz{\infty}
\def\az{\alpha}

\def\lz{\lambda}
\def\dz{\delta}

\def\ez{\epsilon}

\def\kz{\kappa}
\def\bz{\beta}

\def\gz{{\gamma}}

\def\vz{\varphi}

\def\sz{\sigma}

\def\wz{\widetilde}

\def\nb{\nabla}

\def\bint{{\ifinner\rlap{\bf\kern.35em--}
\int\else\rlap{\bf\kern.45em--}\int\fi}\ignorespaces}

\def\bbint{{\ifinner\rlap{\bf\kern.35em--}
\hspace{0.078cm}\int\else\rlap{\bf\kern.45em--}\int\fi}\ignorespaces}

\def\la{\langle}
\def\ra{\rangle}

\newtheorem{thm}{Theorem}[section]
\newtheorem{lem}[thm]{Lemma}
\newtheorem{rem}[thm]{Remark}
\newtheorem{cor}[thm]{Corollary}
\numberwithin{equation}{section}

\begin{document}

\arraycolsep=1pt

\title{\Large\bf Existence of hyperbolic motions to a class of Hamiltonians and generalized $N$-body system via a geometric approach
 \footnotetext{\hspace{-0.35cm}
\endgraf
 2020 {\it Mathematics Subject Classification:} Primary 70F10 $\cdot$ 70H20; Secondary 37J39 $\cdot$ 37J51 $\cdot$ 53C22
 \endgraf {\it Key words and phrases:}  Hyperbolic motion,  $N$-body problem, Hamiltonian, Ma\~{n}\'{e}'s potential.
\endgraf The first author is supported by the Academy of Finland via the projects: Quantitative rectifiability in Euclidean and non-Euclidean spaces, Grant No. 352649, and Singular integrals, harmonic functions, and boundary regularity in Heisenberg groups,
Grant No. 328846.
 The second author is partially funded by NSFC (No. 11871086).
The third author is supported by NSFC (No. 11871088 \& No.12025102) and by the Fundamental Research Funds for the Central Universities.
\endgraf }
}

 \author{Jiayin Liu, Duokui Yan and Yuan Zhou}




\maketitle

\begin{center}
\begin{minipage}{13.5cm}\small
{\noindent{\bf Abstract.}\quad  For the classical $N$-body problem in $\rr^d$ with $d\ge2$,  Maderna-Venturelli in their remarkable paper [Ann.
Math.  2020] proved the existence of hyperbolic motions with any   positive energy constant, starting from any configuration and along any non-collision configuration. Their original proof relies on the long time behavior of solutions by Chazy 1922 and Marchal-Saari 1976, on the H\"{o}lder estimate for  Ma\~{n}\'{e}'s potential by Maderna 2012,  and on the weak KAM theory.

 \quad\quad
We give a new and completely different proof for the above existence of hyperbolic motions. The central idea is that, via some geometric observation, we build up uniform estimates for Euclidean length and angle of geodesics of Ma\~{n}\'{e}'s potential starting from a given configuration and ending at the ray along a given non-collision configuration. 

 \quad\quad
 Moreover, our geometric approach works for Hamiltonians $\frac12\|p\|^2-F(x)$, where $F(x)\ge 0$  is  lower semicontinuous and decreases very slowly to $0$  faraway from collisions. We therefore obtain the existence of hyperbolic motions to such Hamiltonians with any positive energy constant, starting from any admissible configuration and along any non-collision configuration. Consequently,  for several important potentials $F\in C^{2}(\Omega)$,  we get similar existence of hyperbolic motions to the generalized $N$-body system $\ddot{x} = \nb_x F(x)$, which is an extension of Maderna-Venturelli [Ann. Math.  2020].}

\end{minipage}
\end{center}

\tableofcontents
\contentsline{section}{\numberline{ } References}{45}
\medskip

\section{Introduction}\label{s1}

The classical $N$-body problem in $\rr^d$ with $d\ge 2$ {\color{black}consists in studying} the second order ordinary system
\begin{equation}\label{eq1}
  \ddot{x} = \nb_x U(x),
\end{equation}
where  $x=(x_1,x_2,\cdots, x_N) \in \rr^{d N}$, each  body $x_i$ has mass $m_i>0$,
 $U(x)$ denotes the Newtonian potential, and $\nb_x U=(\nabla_{x_i} U)_{1\le i\le N}$ with
$\nabla_{x_i}=\frac1{m_i}\frac{\partial}{\partial x_i}$  being the gradient with respect to the mass scalar product.
  In other words,
$$
U(x) = \sum_{1\le i<j \le N}\frac{m_i m_j}{|x_i-x_j|}, \quad \mbox{and} \quad \nabla_{x_i} U(x)=\sum_{j=1,j\ne i}^{N}\frac{m_j}{|x_j-x_i|^3}(x_j-x_i).$$
The Hamiltonian corresponding to the system \eqref{eq1} is $\frac12\|p\|^2-U(x)$, where the weighted norm {\color{black}$\|p\|$ is defined by} $\|p\|: =(\sum_{i=1}^Nm_i|p_i|^2)^{1/2}$
for $p= (p_1, \dots, p_N) \in \rr^{dN}$.
Below, denote by $\Omega$ the set of non-collision configurations and by $\Sigma$ the set of configurations including some collision, that is,
$$\Omega:= \{x \in \rr^{dN} : x_i \ne  x_j , \, \forall \, 1\le i < j \le N  \}\quad \mbox{and $\Sigma:=\rr^{dN}\setminus \Omega$}.  $$
It is well known that $\Omega$ is a path connected open set (that is, a domain) in $\rr^{dN}$.

The long-time behavior of solutions to the system \eqref{eq1}  has been insistently concerned after the fundamental works by Chazy in 1918; see \cite{c18, c22,ms76, gs17,mv20,dmmy,bm20, fkm21} and the references therein. Indeed, Chazy classified different asymptotic types of solutions for the three-body problem, and moreover, established the continuity of the limit shape; see \cite{c18,c22}.
Later,  these {\color{black} results} were  generalized to the $N$-body problem for any $N\ge4$; see \cite{pol67, pol76, posa70, sa71, wa91, gs17}. {\color{black} Following}  Chazy \cite{c22}, if  $x:[0,\fz)\to\rr^{dN}$  is a solution to \eqref{eq1}  and satisfies $x(t) = ta + o(t)$ as $t \to +\fz$ for some non-collision configuration $a \in \Omega$, we call it a \textit{hyperbolic motion} to \eqref{eq1}  along $a$. The following was proved in \cite{c18,c22,gs17}.

\begin{thm} \label{thm02}
Let   $x:[0,\fz)\to\rr^{dN}$ be a  hyperbolic motion to \eqref{eq1}  along  $a\in\Omega$.
Then we have the following:
  \begin{enumerate}
     \vspace{-0.3cm}   \item[(i)] The solution $x$ has asymptotic velocity $a$, that is, $  \displaystyle \lim_{t\to \fz} \dot{x}(t)=a;$
   \vspace{-0.3cm} \item[(ii)]
Given any $\ez > 0$, there are constants
$t^\ast > 0$ and $\dz > 0$ such that, for any maximal solution $y : [0, T ) \to \Omega$
satisfying $| y(0) - x(0) | < \dz$ and $| \dot y(0) - \dot x(0) | < \dz$, we have

  (ii)$_1$ $T=+\fz$, $| y(t) - ta | < t\ez$ for all $t>t^\ast$ and

  (ii)$_2$ there is $b \in \Omega$ with $|b-a| < \ez$ such that $ y= tb+ o(t).$
  \end{enumerate}
\end{thm}

In 1976, Marchal-Saari  provided the following classification of long-time behavior for the $N$-body problem; see Theorem 1 and Section 10 in \cite{ms76}.
\begin{thm}\label{thm01}
   Let $x : [0, +\fz) \to \rr^{dN}$ be a   solution to \eqref{eq1}.  Then, as $t \to +\fz$, either
    \begin{equation}\label{suhy}
      \frac{R(t)}{t} \to +\fz,  \, \mbox{where $R(t): = \max\{ |x_i(t)-x_j(t)|:1\le i<j\le N\}$,}
    \end{equation}
    or
    \begin{equation}\label{hy1}
\mbox{ there exists $A \in \rr^{dN}$ such that} \      x(t)= At+ o(t).
    \end{equation}
\end{thm}

Note that if $x$ satisfies \eqref{hy1} with $A \in \Omega$, then $x(t)$ is exactly a hyperbolic motion. In general, \eqref{hy1} may not be true;  and even {\color{black} when} it is true,
{\color{black} the condition} $A \in \Omega$ does not necessarily hold. Thus Theorem \ref{thm01} does not affirm the existence of hyperbolic motions.

In 2020, Maderna-Venturelli showed the existence of hyperbolic motions   to \eqref{eq1}  starting from any configuration and along any given non-collision configuration; see  Theorem 1.1 in \cite{mv20}.
\begin{thm} \label{thm0}
 Given any initial configuration $x_0 \in \rr^{dN}$,  any non-collision configuration $a \in \Omega$ with $\|a\| =1$, and  any choice of the energy constant $\lz > 0$, there always exists a hyperbolic motion
$x:[0,\fz)\to\rr^{dN}$ to \eqref{eq1} starting from   $x_0$  and along $\sqrt{2\lz}a$, that is,
$$\mbox{$ x(0)=x_0$ and }\ x(t) = \sqrt{2\lz} a t + o(t) \mbox{ as } t \to +\fz.$$
\end{thm}

The original  proof of  Theorem \ref{thm0} by Maderna-Venturelli     relies   on Theorems \ref{thm02} and \ref{thm01}, and also uses the weak KAM theory and H\"{o}lder estimate for  Ma\~{n}\'{e}'s potential. Indeed,
starting from given $x_0$ and along given $a$ as in Theorem \ref{thm0},
Maderna-Venturelli \cite{mv20}
constructed unbounded calibrating curves of viscosity solutions to Hamilton-Jacobi equation involving the Newtonian potential,
where the possibility of bounded calibrating curves was ruled out by exploiting a classical result by Von Zeipel \cite{z08}(1908). Next, thanks to Marchal-Saari's classification result in Theorem \ref{thm01},
such unbounded calibrating curves satisfy either \eqref{suhy} or \eqref{hy1}. Via some H\"{o}lder estimate for Ma\~{n}\'{e}'s potential $m_\lz$ by Maderna \cite{m12} (see also \cite[Theorem 2.11]{mv20}), the possibility of \eqref{suhy} is  excluded.
Thus such unbounded calibrating curves must satisfy \eqref{hy1} for some $A\in \rr^{dN}$.
Finally, by using the asymptotic behavior in Theorem \ref{thm02} by Chazy \cite{c22} and Gingold-Solomon \cite{gs17},  they showed that $A$ must be $\sqrt{2\lz} a$ as desired.

In this paper, we  give a new and completely different proof for Theorem \ref{thm0}, which is based on some geometric observation for geodesics of Ma\~{n}\'{e}'s potential. 
The central point is that: we consider a solution  $\gz^{(n)} :[0,\sz^{(n)} ]\to \rr^{dN}$  to the system \eqref{eq1} joining  $x_0$ and   $x_0+s^\ast a+ 2^na$ for all sufficiently large $n$, where $s^\ast$ is a constant depending on $x_0$ and $a$; indeed, $\gz^{(n)}$ is a geodesic of Ma\~{n}\'{e}'s potential $m_\lz$. Via some geometric observation  and  an  induction argument, we are able to bound, uniformly in $n$,  the Euclidean length of  $\gz^{(n)} |_{[0,t]}$ and also the angle between $\gz^{(n)} (t)-x_0-s^\ast a$ and $a$ whenever $t\in[\sz _0,\sz ^{(n)}]$ for some  constant $\sz_0$ depending on $x_0,\lz$ and $a$. The desired hyperbolic motion is then given by  the limit of $ \gz^{(n)}$ as $n\to\fz$ (up to some subsequence). See Section 2 for more detail.

Moreover, our geometric approach works for Hamiltonians $\frac12\|p\|^2-F(x)$, where $F(x)\ge0$ is lower semicontinuous and decreases very slowly to $0$  faraway from collisions.
We therefore obtain the existence of hyperbolic motions to such Hamiltonian with any positive energy constant, starting from any admissible configuration, and along any non-collision configuration. Consequently, we get a similar existence of hyperbolic motions to the generalized $N$-body problem $\ddot{x} = \nb_x F(x)$ for several important potentials $F\in C^{2}(\Omega)$, which  is an extension of Maderna-Venturelli [Ann. Math.  2020]. See Theorem 1.4  and Corollary 1.5 below.

To be precise, we  first recall several necessary notions and facts for the Hamiltonian $\frac12\|p\|^2-F(x)$.
In this paper we always assume
\begin{equation}\label{F}
\left\{\begin{split}
& F(x):= {\color{black} \sum _{1\le i<j\le N}} F_{ij}(x_i,x_j)\quad \mbox{$\forall\ x=(x_1,\cdots, x_N)\in\rr^{dN}$}, \mbox{ where }\\
&\mbox{$F_{ij}:\rr^d\times\rr^d\setminus \Delta \to [0,\fz)$ is locally bounded, and is also lower semicontinuous,  that is, }\\
&\mbox{for any $\dz\in\rr$ the level set
$\{(x_i,x_j)\in \rr^d\times \rr^d\setminus\Delta: F_{ij}(x_i,x_j)>\dz\}$ is open.}
\end{split}\right.
\end{equation}
Here and below we write $\Delta=\{(z,z):z\in\rr^d\}$.  In a natural way,  one may extend $F_{ij}$ to be a lower semicontinuous function  in  $\rr^{d}\times\rr^d$  with values in $[0,\fz]$, and hence extend $F$ to be a   lower semicontinuous function in $\rr^{dN}$  with values in $[0,\fz]$. See Appendix for details.

For  any energy constant $\lz> 0$,   Ma\~{n}\'{e}'s potential  or the action potential  $m_\lz$ is defined by
  \begin{align} \label{mane2}
  m_{\lz}(x,y)   : = \inf \{ A_\lz(\gz) : \gz\in \mathcal {AC}  (x, y; [0,\sz] ; \rr^{dN})\,\, \mbox{for some $\sz>0$}\},\qquad \forall  x,y\in\rr^{dN},
\end{align}
where
$$A_\lz(\gz):=\int_0^\sz [\frac12|\dot\gz(s)|^2+F\circ\gz(s)+\lz] \, ds, $$
 and for any set $W\subset\rr^{dN}, $
 $$ \mathcal {AC}  (x, y; [0,\sz] ; W): = \{ \gz:[0,\sz] \to W \ \mbox{is absolutely continuous, and } \gz(0)=x, \gz(\sz)=y\}.$$

Since $\Omega$ is path-connected,  $(\Omega,m_\lz)$ {\color{black} is} a metric space. {\color{black} Denote by $\wz\Omega$ the set
\begin{equation}\label{wzoz}
  \{y \in \rr^{dN}: \text{there exists $\lz>0$ and $x \in \Omega$ such that $m_\lz(x,y)<+\fz$}\}.
\end{equation}
We remark that the set $\wz\Omega$ is independent of the choice of $\lz>0$ and $x\in \Omega$; see Appendix}. Configurations $x\in\wz\Omega$ are also {\color{black} said} to be   admissible.
It turns out that $(\wz\Omega,m_\lz)$ is a geodesic space, and  geodesics joining any pair of distinct configurations $x,y\in\wz\Omega$   are  given by the arc-length  parametrization of  minimizers of $A_\lz$ in the class
$\cup_{\sz>0}\mathcal {AC}  (x, y; [0,\sz] ; \wz\Omega)$; see Appendix for more details.
Below, to
 clarify the geometric nature, we call a minimizer  of $A_\lz$
as an \textit{$m_\lz$-geodesic with canonical parameter}. We also call a  continuous curve $\gz:[0,\fz)\to\wz\Omega$ as an \textit{$m_\lz$-geodesic ray  with canonical parameter}    if,
for any $\sz>0$, the restriction
$\gz|_{[0,\sz]}$ is an   $m_\lz$-geodesic  with  canonical parameter.
Moreover, in the spirit of Chazy,  if $\gz$ is an $m_\lz$-geodesic ray with canonical parameter for some $\lz>0$ and satisfies $x(t) = at + o(t)$ as $t \to +\fz$  for some non-collision configuration $a$, we  call $\gz$ as a \textit{hyperbolic motion  to the Hamiltonian  $\frac12\|p\|^2-F(x)$ with energy constant $\lz>0$  and along $a$,} or for simplicity, \textit{hyperbolic motion to $m_\lz$  along $a$}.

Next, assuming  in addition that  $ F_{ij}\in C^{2}( \rr^d\times \rr^d\setminus\Delta ;  \rr )$  for all $1\le i<j\le N$ (for simplicity we write $ F \in C^{2}( \Omega)$ below), we consider the generalized $N$-body system
\begin{equation}\label{eq4}
  \ddot{x} = \nb_x F(x),
\end{equation}
which corresponds to the Hamiltonian  $\frac12\|p\|^2-F(x)$.
Following Chazy, if  a solution $x:[0,\fz)\to\rr^{dN}$ to \eqref{eq4} satisfies $x(t) = at+ o(t)$ as $t \to +\fz$  for some non-collision configuration $a \in \Omega$, then we call it as a \textit{hyperbolic motion  to \eqref{eq4} along $a$.}
Recall that if $\gz:[0,\sz]\to\wz\Omega$ is an  $m_\lz$-geodesic with canonical parameter and if $\gz$ is collision free interiorly, that is, $\gz|_{(0,\sz)}\subset\Omega$,   then it is a solution to the system \eqref{eq4} with energy constant $\lz$,
see Lemma A.3 and Lemma A.4 in Appendix. Consequently, a hyperbolic motion to $m_\lz$ without collision in $(0,\fz)$ is a hyperbolic motion to \eqref{eq4}.

Below, we build up an existence {\color{black} result} of hyperbolic motions to the Hamiltonian $\frac12\|p\|^2-F(x)$, and hence to the corresponding N-body system $\ddot{x} = \nb_x F(x)$ when $F\in C^2(\Omega)$ additionally,
provided that $F$ satisfies the following growth assumption faraway from collisions:
\begin{equation}\label{sapa1}
\left\{\begin{split}
& F_{ij}(x_i, x_j) \le  f(|x_i-x_j|), \ \mbox{ $\forall (x_i,x_j)\in\rr^{d}\times\rr^d\setminus \Delta_{2}$    for all  $1\le i<j \le N$,}\\
&\mbox{ where   $f \in C^{0}(\rr_+,\rr_+)$  satisfies}\quad \sum_{k=1}^\fz\left( 2^{-k}\int_{2^k}^{2^{k+1}}f(s)\,ds\right)^{1/2}<\fz.
\end{split}\right.
\end{equation}
Here and below $\rr_+:=(0,\fz)$, and
  $\Delta_{\dz}:=\{(z,w)\in\rr^{d}\times \rr^d : |z-w|< \dz\}.$
 Note that, if $f$ is decreasing, then
  $$  \int_4^{\infty}\sqrt{f(s)}\frac{ds}s\le \sum_{k=2}^\fz \left( 2^{-k}\int_{2^k}^{2^{k+1}}f(s)\,ds\right)^{1/2} \le 2\int_2^{\infty}\sqrt{f(s)}\frac{ds}s.$$
Moreover, if
{\color{black} $ f(s)\le C(\log s)^{-\beta}$,}    $  \forall  s\ge 2$, for some  $\beta>2$, then
$$ \sum_{k=1}^\fz \left( 2^{-k}\int_{2^k}^{2^{k+1}}f(s)\,ds\right)^{1/2}\le {\color{black} \frac{C^{1/2}}{(\log 2)^{\bz/2}}}\sum_{k=1}^\fz k^{-\beta/2}<\fz.$$

\begin{thm}\label{mthm} Let $F$ be as in \eqref{F}  and {\color{black}satisfying } \eqref{sapa1}.
Given any energy constant $\lz>0$, any initial configuration $x  \in\wz\Omega$ and any non-collision configuration $a \in \Omega$ with $\| a \| = 1$, we have the following:
\begin{enumerate}
\item[(i)] There exists a hyperbolic motion $\gz : [0, +\fz) \to \wz\Omega$
 to   the Hamiltonian $\frac12\|p\|^2-F(x)$  with the energy constant $\lz>0$, along $\sqrt {2\lz}a$ and
starting from $x$, that is,
$$\mbox{ $\gz (0)=x  $ and
 $\gz(t) = \sqrt{2\lz} a t + o(t) $  as $t \to +\fz.$}$$

\item[(ii)]  There exists  $t_0=t_0(x,a,\lz)\ge 0$ such that the restriction
$\gz|_{[t_0,\fz)}$ is collision-free, that is,
  $\gz|_{[t_0,\fz)}\subset\Omega$.

\item[(iii)] If $ F \in C^{2}( \Omega)$,
  then the restriction  $\gz|_{[t_0,\fz)}$ is a hyperbolic motion to \eqref{eq4} along $\sqrt {2\lz}a$.

\item[(iv)] If  $ F \in C^{2}( \Omega)$,  and suppose in addition  $\gamma|_{(0,t_0)}$ is collision free,
 then   $\gz $ is a hyperbolic motion to \eqref{eq4}
starting from $x$ and along $\sqrt {2\lz}a$.
\end{enumerate}
\end{thm}

The study of (non-)collisions is one of the main issues in this field. To some extent,  (non-)collisions are determined by the behavior of $F$ near  the set $\Sigma$ of collisions. There do exist potentials $F$, for which  one can find some $m_\lz$-geodesic including interior collisions, see \cite{btv13,btv14,mo18} for constructions. Therefore,  in Theorem 1.4 (iv), since
no other assumption is made on $F$ around $\Sigma$, we can not expect in general that the restriction $\gz|_{(0,t_0)}$ is collision-free.

On the other hand, it is well known (see \cite{ft04, bft08}) that all $m_\lz$-geodesics are collision-free interiorly, and hence are solutions to the system \eqref{eq4}, provided that  $F(x)$ satisfies any one of the following conditions  near $\Sigma$:
 \begin{enumerate}
\vspace{-0.3cm}
\item[(E1)]  For all $1\le i< j\le N$,
 $F_{ij}(x_i, x_j)\ge C m_i m_j |x_i-x_j|^{-2}$   in $\Delta_{1/2}$,
where  $C>0$.

\vspace{-0.3cm}\item[(E2)]    For all $1\le i< j\le N$,  $F_{ij} $ coincides with  the homogeneous potential $m_i m_j |x_i-x_j|^{-\alpha}$
in $\Delta_{1/2}$, where   $0<\alpha<\fz$.  When $\alpha=2$,
it was also called as the Jacobi-Banachiewitz potential; see \cite{ba106,w138, ch198}.

\vspace{-0.3cm}
\item[(E3)]  For all $1\le i< j\le N$,
 $F_{ij} $ is given by   the quasi-homogeneous potential
$$m_i m_j \left(|x_i-x_j|^{-\alpha}+\delta |x_i-x_j|^{-\beta} \right) \mbox{in $\Delta_{1/2}$, where  $0<\beta<\alpha<\fz$ and $\delta>0$.}$$
It was also called as the Manev potential when $\alpha=2$ and $\beta=1$ (see   \cite{ma124,di193, lm12}),
and as
the Schwarzschild potential
 when $\alpha=3$ and $ \beta=1$  (see \cite{aps14}).

\vspace{-0.3cm}
\item[(E4)] For all $1\le i< j\le N$,  $F_{ij} $ is given by the
 Lennard-Jones potential (see \cite{lj131, ll15}), that is, $A|x_i-x_j|^{-12}-B|x_i-x_j|^{-6}$  in
$\Delta_{\dz}$ for some $\dz=\dz(A,B)>0$, where  $A, B >0$.

\vspace{-0.3cm}
\item[(E5)] For all $1\le i< j\le N$,  $F_{ij} $
is gien by the Seeliger-Yukawa potential (see \cite{se1895, vn18}), that is, $A e^{-B|x_i-x_j|} |x_i-x_j|^{-1}$ in $\Delta_{1/2}$, where  $A, B >0$.
\vspace{-0.3cm}
\item[(E6)]   For all $1\le i< j\le N$,
$F_{ij} $ is given by the M\"{u}cket-Treder potential (see \cite{mt177, pp20}), that is,  $ |x_i-x_j|^{-1}(A-B  \log|x_i-x_j|)$ in $\Delta_{1/2}$,
where  $A, B >0$.

\vspace{-0.3cm}
\item[(E7)]  For all $1\le i< j\le N$, $F_{ij}$ coincides with the logarithmic  potential $m_i m_j \log|x_i-x_j|^{-1}$
in $\Delta_{1/2}$.
\end{enumerate}

Thus, as a direct consequence of Theorem 1.4 (iv), we have the following.

\begin{cor}\label{mthmc} Let $F\in C^{2}(\Omega)$ be as in \eqref{F} and satisfy \eqref{sapa1}.
Suppose that $F$ satisfies any one of (E1)-(E7) near $\Sigma$.

Given any energy constant $\lz>0$, any initial configuration $x  \in\wz\Omega$ and any non-collision configuration $a \in \Omega$ with $\|a\|=1$,   there is a hyperbolic motion   to the system \eqref{eq4} with the energy constant $\lz$,  along $\sqrt{2\lz} a$  and  starting from $x$.
\end{cor}

In Section 3--Section 8, we prove Theorem \ref{mthm} and hence Corollary 1.5;
in particular, when $\mbox{$F=U$ with $f(s)=s^{-1}$}$,  we  reprove Theorem 1.3.
For reader's convenience, in Section 2 we sketch our geometric proof to Theorem \ref{mthm}, and also introduce some necessary notations used in Section 3--Section 8 and Appendix.

\begin{rem}\rm
We remark that it seems possible to extend our geometric approach to  general Hamiltonians $H(x,p)$, and also  to the setting of (sub)Riemannian manifolds. We would like to see such  extension.
\end{rem}

{\color{black}
\medskip
\noindent {\bf Acknowledgement.}
The authors would like to thank the anonymous referee  for the careful reading,  several valuable suggestions/comments, and many important and detailed corrections,  which significantly improve the final presentation of the paper.  In particular, the authors thank the referee for   pointing out
a wrong statement  in the original version  about the relationship between $ (\wz\Omega,m_\lz)$ and the completion of $(\Omega,m_\lz)$.
 In this revision,  we correct this wrong statement in  Remark A.1 (ii) and (iii).
 In Remark A.1 (ii), under the assumption  that the potential $F=+\fz$ in $\Sigma$, we do show that $ (\wz\Omega,m_\lz)$  is the completion of $(\Omega,m_\lz)$.
  However, in Remark A.1 (iii),   we construct a potential $F$,  for which
  the set
  $ \{F=+\fz\}$ is strictly contained in $\Sigma$,  so that
the completion of $(\Omega,m_\lz)$ is strictly contained in $(\wz\Omega,m_\lz)$.
We are really in debt to the  anonymous referee.
}

\section{Sketch  of  the  proof to Theorem \ref{mthm} }

Note that we only prove Theorem \ref{mthm} under the assumption that
\begin{equation}\label{assmass}\mbox{the mass $m_i=1$ for each $1\le i\le N$.} \end{equation}
For general masses, by some necessary and standard modifications, our argument also works; we omit the details.

Before  sketching the proof, we introduce several notations which are needed below. In the sequel of this paper, we always assume that \eqref{assmass} holds and hence the corresponding weighted norm $\|\cdot\|$
(resp. weighted inner product)
{\color{black} is} then the Euclidean norm (resp. inner product), that is,
\begin{equation}\label{ass1} \|x\|=(\sum_{i=1}^N|x_i|^2)^{1/2}=|x|\quad
\mbox{and hence }\quad \langle \langle x,y\rangle\rangle=\sum_{i=1}^N\langle x_i,y_i\rangle\quad\mbox{ for $x,y\in\rr^{dN}$.}
\end{equation}

To measure the distance of any configuration $y\in\rr^{dN}$ to the set of collisions,   we always set
\begin{equation}\label{ass2}   y^\flat:=\min\{|y_i-y_j|: 1\le i<j\le N\}.
\end{equation}
Note that $y^\flat>0$ if and only if $y\in\Omega$, that is, $$\Omega=\{y\in \rr^{dN}:y^\flat>0\}.$$

For any $x,y\in\rr^{dN}$,
write $[x,y]$ (resp. $(x,y)$)  as the closed  (resp. open) {\color{black}line-segment} between them,  that is,
\begin{equation}\label{ass3} [x,y]=\{tx+(1-t)y\}_{t\in[0,1]} \ \mbox{and}\ (x,y) =[x,y]\setminus\{x,y\}=\{tx+(1-t)y\}_{t\in(0,1)}.
\end{equation}
{\color{black}Write $\rr_+=(0,\fz)$, the open half ray
$$x+\rr_+y=\{x+ty:t\in\rr_+\}$$
and $\overline{x+\rr_+y}$ as the closure of $x+\rr_+y$.}
{\color{black}Given $A \subset \rr^{dN}$ and $z \in \rr^{dN}$, let us denote  $d_E(z,A)$ the Euclidean distance from $z$ to $A$. In particular,   $d_E(z,[x,y])$ is the  Euclidean distance from $z$ to the  line-segment $[x,y]$, and $d_E(z,x+\rr_+y)$ is the  Euclidean distance from $z$ to the open half ray $x+\rr_+y$.}

For any absolutely  continuous curve $\gz:[0,\sz]\to\rr^{dN}$,   denote by $l(\gz)$  its Euclidean length, that is,
\begin{equation}\label{ass4}
l(\gz)=\int_0^\sz|\dot \gz(s)|\,ds.
\end{equation}
For any two configurations $a,b\in\rr^{dN}$ with $a,b\ne0$,  denote by $\angle (a,b)$ the angle between them, {\color{black}that is the unique element of $[0, \pi]$ satisfying}
$$\cos\angle(a,b)=\frac{\langle a,b\rangle }{|a||b|}.$$
Denote by $\mathbb S^{dN-1}$ {\color{black} the set of all configuration $a\in \rr^{dN}$ satisfying $|a|=1$}.
Note that, for any two configurations $a,b\in\mathbb S^{dN-1}$, the smallness of
 $|b-a| $  is  equivalent to that  of the angle $\angle(a,b)$; indeed,
\begin{equation}\label{angle=dis}\frac12|a-b|=\sin\frac12\angle(a,b).
\end{equation}

Now we are ready to sketch the proof to Theorem \ref{mthm}
  and hence  Corollary 1.5,
in particular, when  $\mbox{$F=U$ and $f(s)=s^{-1}$}$, we {\color{black}sketch  a new proof of}  Theorem 1.3.

\begin{proof}[Sketch of the proof to Theorem \ref{mthm}]
Given any energy constant $\lz>0$, any initial configuration $x \in\wz\Omega$ and any non-collision configuration $a\in\Omega$ with $|a|=1$, we obtain the desired hyperbolic motion in  Theorem \ref{mthm} by the following 4 steps.

\medskip
\noindent {\bf Step 1.}
Writing
 $x^\ast:= x+  50(1+|x|)a/a^\flat,$
we have
 \begin{equation}\label{zz1a}
\mbox{$(x^\ast +tb)^\flat\ge 2$ and hence  $x ^\ast  +tb\in \Omega$ whenever $t\ge 0$ and
 $ b\in \mathbb S^{dN-1}$ with $|b-a|\le \frac1{20}a^\flat$.}
\end{equation}
For any $t, b$ as in \eqref{zz1a}, and
for   any  $m_\lz$-geodesic $\gz$  with canonical parameter  joining $x$ and   $x^\ast +tb$, we  have the following upper bounds:
\begin{equation}\label{eqk1a}
  t\le l(\gz ) \le \frac1{\sqrt{2\lz}}m_\lz(x,x ^\ast  +tb) \le  t+  \Psi(t), 
\end{equation}
where
 $$
 \Psi(t):=\frac1{\sqrt{2\lz}}m_\lz(x,x^\ast)+\frac1\lz N^2\frac1{a^\flat}\int_{(x^\ast)^\flat}^{2t+2|x|}f(s)\,ds.$$

Both \eqref{zz1a} and \eqref{eqk1a}  are stated in Section 5.
Their proofs  are given in Section 3 and Section 4  based on the  basic geometry of non-collision configurations  established in Section 3 and some careful analysis.
The choice of the {\color{black}configuration} $x^\ast$  is necessary in general  to    get \eqref{eqk1a} since   $F$ is only assumed to be bounded by $f$ faraway from collisions.

 We remark that  the growth assumption \eqref{sapa1} guarantees that
\begin{equation}\label{psit} \wz \Psi(t)=\sum_{j\ge [\log_2 t ]}\sqrt{2^{-j}\Psi(2^j)} \to 0\quad\mbox{as $t\to\fz$};\end{equation}
see Section 5. Below we fix $n_0\in\nn$ large enough,  in particular, to guarantee that $\wz\Psi(2^{n_0})\le 2^{-10}a^\flat$.

\medskip
\noindent{\bf Step 2.}    For any integer $n \gg n_0$, denote by   $\gz ^{(n)}:[0,\sz^{(n)}]\to\wz \Omega$  an $m_\lz$-geodesic with canonical parameter joining $x$ and $x^\ast
+2^na$.  In Section 5, we prove the following crucial geometric observation: If $$\mbox{ $x^\ast + Sb= \gz^{(n)} ( \tau) $  for some  $b\in \mathbb S^{dN-1}$ with $ |a-b |\le \frac1{20}a^\flat$, and  $S\ge 2^{n_0}$ and $ \tau\in (0,\sz^{(n)}]$},$$then one has
 \begin{align} \label{pw1a}
 |\gz ^{(n)}(  \tau_{1/2}) -(x^\ast + \frac S2b)| &\le 2 \sqrt{  S \Psi(S)} \mbox{ with}\
   \tau_{1/2} :=\max\{0<t< \tau: |x^\ast-{\color{black}\gz^{(n)} }(t)|=\frac S2\},
\end{align}
 and 
\begin{equation}\label{zq2a}
  d_E(\gz^{(n)}(t), x^\ast+ \rr^+b )  \le  4\sqrt {S \Psi(S)}, \quad\forall t\in[ \tau_{1/2},\tau].
\end{equation}

To see them, first note that the assumptions in $S$ and $\tau$ allow us to
conclude $l (\gz^{(n)}|_{[0,\tau]} )\le S+\Psi(S)$ from
  \eqref{eqk1a}.  Next, we construct  some suitable  triangles as illustrated by Figure 1--Figure 5 in Section 5,  and then bound the lengths of the longest sides via the Euclidean lengths of certain restrictions of $\gz^{(n)}$, which will be estimated with the aid of  $l (\gz^{(n)}|_{[0,\tau]} )$.
Using the geometry of   triangles, we could get the desired \eqref{pw1a} and \eqref{zq2a} through careful calculation.

\medskip

\noindent {\bf Step 3.}
Thanks to \eqref{eqk1a},  \eqref{pw1a} and \eqref{zq2a}, by  an induction argument as illustrated in Figure 6 of Section 6,  we obtain the following upper bound  for the Euclidean length  
\begin{align}\label{sizeesta}l(\gz^{(n)}|_{[0,t]})\le \frac1{\sqrt{2\lz}}m_{\lz}(x,\gz^{(n)}(t))\le |\gz^{(n)}(t)-x^\ast|(1+\wz\Psi^2(|\gz^{(n)}(t)-x^\ast|))   \quad  \mbox{whenever}\
\sz^{(n)}_{n_0}\le t<\sz^{(n)}
\end{align}
with     $$ \sz^{(n)}_{n_0}:=\max\{s\in[0,\sz^{(n)}]:
|\gz^{(n)}(s)-x^\ast|=2^{n_0}\},$$  and also obtain
an  upper bound  for the angle $\angle(\gz^{(n)} (t)- x^\ast,a)$  between  $\gz^{(n)} (t)- x^\ast$ and $a$,  equivalently,
\begin{align}\label{distanceesta}| \frac{\gz^{(n)}(t)-x^\ast}{|\gz^{(n)}(t)-x^\ast|}-a | \le 2^4\wz\Psi( |\gz^{(n)}(t)-x^\ast|) \quad  \mbox{whenever}\ \sz^{(n)} _{n_0}\le t<\sz^{(n)} .
 \end{align}
 See Section 6 for more details.

\medskip

\noindent {\bf Step 4.} Recall  that $\gz^{(n)}$ has energy constant $\lz$, that is,
$$|\dot\gz^{(n)}(s)|^2=2
F\circ\gz^{(n)}(s)+2\lz \quad\mbox{for almost all } s\in[0,\sz^{(n)} ];
$$
see Lemma A.3 in Appendix. Combining this, \eqref{sizeesta} and \eqref{distanceesta},  we prove in Section 7  that
$$\frac{|\gz^{(n)}(t)-x^\ast-\sqrt{2\lz}at|   }{  \sqrt{2\lz}t }\le 2^5  \wz \Psi(\frac12\sqrt{2\lz}t)  \quad  \mbox{whenever}\  {\color{black}\sz^{(n)}} \ge t\ge \frac1{\sqrt{2\lz}}2^{n_0+1}.$$

Next, sending $n\to\fz$ (up to some subsequence),  $\gz^{(n)}$ converges to some
 $m_\lz$-geodesic ray $\gz$ with canonical parameter,  which  satisfies
$$\frac{|\gz (t)-x^\ast-\sqrt{2\lz}at|   }{  \sqrt{2\lz}t }\le 2^5  \wz \Psi(\frac12\sqrt{2\lz}t) \quad  \mbox{whenever}\  t\ge \frac1{\sqrt{2\lz}}2^{n_0+1}.$$
Recall \eqref{psit}, that is, $\wz\Psi(s)\to 0$ as $s\to\fz$. We know that $\gz$ is the desired hyperbolic motion starting from $x$ and along $\sqrt {2\lz}a$. For details see Section 8.

\end{proof}

\section{Some basic geometric properties of non-collision configurations}\label{s2}

In what follows, we always assume the mass $ m_i=1$ for each $1\le i\le N$. Recall that $\Omega$ is the set of non-collision configurations, that is, all $x\in\rr^{dN}$ with $x^\flat>0$. We refer to \eqref{ass1}-\eqref{ass4} and therein for several notations which will be used below.

In this section, we prove several basic geometric properties of non-collision configurations (see Lemma 2.1--Lemma 2.4),  which will be used later.

\begin{lem}\label{angle}
Let  $b \in \rr^{dN}$ and $a\in \bs^{dN-1}$ with $a^\flat>0$.
If $|a-b|\le \dz{a^\flat} $ for some $0<\dz<\frac15$, then
\begin{equation}\label{bflata}
b^\flat > a^\flat-2|b-a|\ge  (1-2\dz)a^\flat,
\end{equation}
and
\begin{equation*}
\cos\angle (a_i-a_j,b_i-b_j) =\frac{\langle a_i-a_j,b_i-b_j\rangle }{|a_i-a_j||b_i-b_j|} \ge  1-4\dz, \quad\forall 1\le i<j\le N.
\end{equation*}
If we further assume $b\in \bs^{dN-1}$, then
\begin{equation}\label{angleab}
\cos\angle (a,b)=\langle a,b\rangle\ge 1-2\dz^2.
\end{equation}
\end{lem}
\begin{proof} By the triangle inequality,
$$|b_i-b_j|\ge |a_i-a_j|- |b_i-a_i|-|b_j-a_j|> |a_i-a_j|-2|b-a|, \quad\forall 1\le i<j\le N.$$
From this, we conclude
$$b^\flat> a^\flat-2|b-a|.$$
Thus, by $|a-b|\le \dz a^\flat $, it holds that
$$b^\flat \ge (1-2\dz)a^\flat.$$

Moreover, for all $ 1\le i<j\le N,$  by the Cauchy-Schwartz inequality and the triangle inequality,
$$\langle a_i-a_j,b_i-b_j\rangle= \langle a_i-a_j,a_i-a_j\rangle+\langle a_i-a_j,b_i-a_i-b_j+a_j\rangle\ge |a_i-a_j|[|a_i-a_j|-2|b-a|]. $$
Note that
$$|b_i-b_j|\le |a_i-a_j|+ |b_i-a_i|+|b_j-a_j|\le |a_i-a_j|+2|b-a|.$$
We then have
$$\frac{\langle a_i-a_j,b_i-b_j\rangle}{|a_i-a_j||b_i-b_j|} \ge  \frac{ |a_i-a_j|-2|b-a| }{|a_i-a_j|+2|b-a|}= 1- \frac{4|b-a| }{|a_i-a_j|+2|b-a|}\ge 1-\frac{4|b-a| }{a^\flat },$$
  and hence, by $|a-b|\le \dz a^\flat $, it holds that
$$\cos  \angle (a_i-a_j,b_i-b_j) \ge    1- 4\dz.$$

For the last inequality \eqref{angleab}, note that $a^\flat<2$ and
\[2-2\langle a,b\rangle = |a-b|^2 \le \delta^2 (a^\flat)^2. \]
It implies
\[ \langle a,b\rangle=1-\frac12 |a-b|^2 \ge 1- \frac12 \delta^2 (a^\flat)^2 \ge 1-2\delta^2 \]
as desired.
\end{proof}

\begin{lem}\label{findxast} Let  $x\in \rr^{dN}$ and
 $a\in \bs^{dN-1}$ with $a^\flat>0$.  For {\color{black}$s\ge50\frac{1+|x| }{a^\flat} $}, we have
$$ |
\frac{x+sa }{|x+sa|}-a| 
\le \frac{ a^\flat }{20}  $$
and $$(x+sa)^\flat\ge \frac{24}{25} s a^\flat>2. $$
\end{lem}
\begin{proof}
For $s>|x|$, write
$$  |
\frac{x+sa }{|x+sa|}-a|= \frac1{|x  + sa|}\left |x + sa-|x + sa|a\right|.$$
By the triangle inequality, one has $$\left |x + sa-|x + sa|a\right|\le |x|  + \left|s- |x    +sa|\right|= |x|  + \left||sa|- |x    +sa|\right|\le 2|x|,$$
and hence
$$ |\frac{x+sa }{|x+sa|}-a| \le \frac{1}{|x+sa|}2|x| \le
\frac1{s-|x|  }  {2|x|}  .$$
By  $s \ge 50\frac{1+|x| }{a^\flat}$ and $a^\flat < 2$,
we have   $s-|x| \ge 40\frac{|x|}{  a^\flat }$ and hence
$$ \frac1{s-|x|  }  {2|x|}\le  \frac{a^\flat}{20}  $$
as desired.

Moreover, it is easy to see that
$$(x+sa)^\flat \ge sa^\flat -2|x|\ge  \frac {24}{25} s  a^\flat.$$
The proof is complete.
\end{proof}

\begin{lem}  \label{xast2a}
If  $z,b\in\Omega$  satisfy
  $$ \angle (b_i-b_j,z_i-z_j)\le \frac\pi2 \mbox{ or  equivalently}\ \cos  \angle (b_i-b_j,z_i-z_j)\ge 0 \quad\forall 1\le i<j\le N,
$$
then  $(z+tb)^\flat>z^\flat>0$, and hence $ z+tb \subset\Omega$, for all $t\in[0,\fz)$.
\end{lem}

\begin{proof}
It suffices to prove that
for any $1\le i<j\le N$,
$$\mbox{$ |z_i+tb_i-(z_j+tb_j)|\ge |z_i-z_j|$ for all $t> 0$. }$$
To see this, write $$\mbox{$q=z_i-z_j$ and $e=b_i-b_j$, and hence, $|z_i+tb_i-(z_j+tb_j)|=|q+te|$}.$$
 Then
  $  |e|\ge b^\flat>0$ and $|q|\ge z^\flat >0$. The assumption  $\angle (q,e)\le \pi/2$ implies $\langle q,e\rangle\ge 0.$  Thus
$$ |q+et|=\sqrt{\langle q+et,q+et\rangle}=\sqrt{|q|^2+2\langle q,e\rangle t+|e|^2t^2}\ge
\sqrt{|q|^2 +|e|^2t^2}\ge
|q|\ge z^\flat> 0, \quad \forall t\ge 0 $$
 as desired.
\end{proof}

\begin{lem}\label{lem3-4}
 Let  $x\in \rr^{dN}$ and
 $a\in \bs^{dN-1}$ with $a^\flat>0$.
If $b\in \bs^{dN-1}$ with $|a-b|\le \frac{a^\flat}{20}$,
then for all  {\color{black}$s\ge 50\frac {1+ |x|}{a^\flat}$} and $ t\ge 0,$ one has
\begin{equation}\label{cofree}
  (x+sa+tb)^\flat\ge (x+sa)^\flat \ge \frac{24}{25} s a^\flat\ge2 .
\end{equation}
Therefore, $x+sa+tb \in \Omega$.
\end{lem}

\begin{proof}
For any fixed $s\ge 50\frac{1+|x| }{a^\flat}$, write $c= \frac{x+sa}{|x + sa|}$.
By Lemma \ref{findxast},  $|c-a|\le \frac{a^\flat}{20}$.
Since $|a-b|\le \frac{a^\flat}{20}$, by \eqref{bflata} in Lemma \ref{angle}, we have
$$|c -b|\le |c-a|+|a-b|\le \frac{a^\flat}{10}\  \, \, \mbox{and}\  \, \,  b^\flat\ge  \frac9 {10}a^\flat >0.$$
Thus
 $|c-b|\le \frac{b^\flat}{9}$. By Lemma \ref{angle}, it implies that  $\angle (c,b) \le  \arccos \frac{79}{81} <\pi/2$. Then for all $\dz\ge 0$, it holds that $\angle(   c+\dz b,b)\le \angle (c,b)<\pi/2$. Hence,
\begin{equation}\label{cbdelta}
 |\frac{c+\dz b }{|c+\dz b |} -b| =2\sin \frac12\angle(   c+\dz b,b)\le 2\sin \frac12\angle (c,b)=
|c-b|\le  \frac{b^\flat}{9}, \,  \quad\forall \dz\ge 0.
\end{equation}
Note that for any $t\ge0$,
\[\frac{x +sa+tb }{|x+sa+tb|}= \frac{c+\frac{t}{|x +sa|}b }{|c+\frac{t}{|x +sa|}b |}. \]
By \eqref{cbdelta}, it follows that
$$ |\frac{x +sa+tb }{|x+sa+tb|} -b|= |\frac{c+\frac{t}{|x +sa|}b }{|c+\frac{t}{|x +sa|}b |}-b|\le    \frac{b^\flat}{9}. $$
By Lemma \ref{angle}, it holds that
$$\cos \angle ( x_i +sa_i+tb_i-x_j -sa
 _j-tb_j,b_i-b_j)\ge \frac 5 9> \frac12,   \quad \mbox{ $\forall$ $1\le i<j\le N$}.$$
From Lemma \ref{xast2a} and Lemma \ref{findxast}, it follows that  $$(x+sa+tb)^\flat\ge   (x+sa)^\flat\ge \frac{24}{25} s a^\flat.$$
We {\color{black} have completed} the proof.
\end{proof}

%
%

\section{An  upper bound of Euclidean length of $m_\lz$-geodesics}
In this section, we build up the following upper bound of the Euclidean length of $m_\lz$-geodesics with canonical parameter.

\begin{lem} \label{uppm} Let $x\in\wz\Omega$, $a\in \mathbb S^{dN-1}$  with $a^\flat>0$ and $\lz>0$.
 For any {\color{black}$s \ge50\frac{1+|x|}{a^\flat}$}, $b\in \mathbb S^{dN-1}$ with $|a-b|\le \frac{a^\flat}{20}$,
and  any $T\ge 0$, if   $\gz\in \mathcal {AC}(x, {\color{black} x+sa+Tb};[0,\sz];\wz\Omega)$  is an  $m_\lz$-geodesic with canonical parameter, then we have
\begin{equation}\label{eqk1}
  l(\gz ) \le \frac1{\sqrt{2\lz}}m_\lz(x,x+sa+Tb) \le  T+ \frac1{\sqrt{2\lz}} m_{\lz}(x,x+sa)+ \frac{1}{ \lz }  N^2
 \frac1{a^\flat  } {\color{black} \int_{( x+sa )^\flat}^{2T+2| x+sa |} f(\tau)  \, d\tau}.
\end{equation}
\end{lem}

To prove this we begin with the  following trivial upper bound.
\begin{lem}  \label{klem}
If
 $\gz\in \mathcal {AC}(x,y;[0,\sz];\wz\Omega)$  is an  $m_\lz$-geodesic with canonical parameter for some $ \lz>0$,
 then  one has
   \begin{equation}\label{simple}
   l(\gz )\le \frac1{\sqrt{2\lz}} m_{\lz}(x,y)\le   |y-x|+ \frac1{2\lz}\int_{0}^{|y-x|}F(x+s\frac{y-x}{|y-x|} )\,ds.
\end{equation}
 \end{lem}
\begin{proof} {\color{black} Concerning the the first inequality} in \eqref{simple}, by the Cauchy-Schwarz inequality one has
$$ \frac 12  |\dot \gz (s)|^2+ F(\gz(s))+\lz\ge \frac 12  |\dot \gz (s)|^2+ \lz
\ge 2\sqrt {\lz} \sqrt{\frac 12  |\dot \gz (s)|^2} =
\sqrt{2\lz }     |\dot \gz (s) |,\quad a.\,e.$$
 Thus
 $${m_\lz(x,y)}= A_\lz(\gz ) \ge \sqrt{2\lz} l(\gz) .$$

To see the second inequality in \eqref{simple}, we set $a= \frac{y-x}{|y-x|}$ and $\eta(s)=x+\sqrt {2\lz}s  a$ for $s\in[0,\frac{|y-x|}{\sqrt {2\lz}}]$.
 By definition, ${m_\lz(x,y)}\le A_\lz(\eta)$. Since
\begin{align*}
  A_\lz(\eta)&=\int_0^{\frac{|y-x|}{\sqrt{2\lz}}}
[\lz+F(x+\sqrt {2\lz}s  a)]\,ds+\lz  \frac{|y-x|}{\sqrt{2\lz}}= \sqrt{2\lz}|y-x|+ \frac1{\sqrt{2\lz}}\int_{0}^{|y-x|}F(x+sa)\,ds,
\end{align*}
we obtained the second inequality in \eqref{simple} as desired.
\end{proof}

\begin{rem}\rm
Note that above  if $[x,y]\subset\Omega$, then $\int_{0}^{|y-x|}F(x+sa)\,ds<\fz$; otherwise, it may happen that
$\int_{0}^{|y-x|}F(x+sa)\,ds=\fz$. Recall that no assumption is made on the behavior of $F$ in $\Sigma$ essentially.
Below  via some necessary geometric argument in Section 3, we could only meet the the integral {\color{black} of type $\int_{0}^{A}F(y+sb)\,ds$ where $A>0, y \in \Omega, b \in \mathbb S^{dN-1} \cap \Omega$ and the line-segment $[y,y+Ab] \Subset \Omega$.}
\end{rem}

Next we bound   $\int_{0}^TF(z+sa)\,ds$ for some $z$ and $a$ with  $\overline{z+\rr_+a}\subset\Omega$.
\begin{lem}  \label{xast2}
Let  $ \theta\in(0,\pi/2)$,
 $z\in \Omega$ with $z^\flat\ge 2$ and $a\in \mathbb S^{dN-1} $ with $a^\flat>0$.
If  $$\sup_{1\le i<j\le N}\angle (a_i-a_j,z_i-z_j)\le \theta, 
$$
then \begin{equation}\label{h1}
\int_{0}^TF(z+sa)\,ds \le   \frac{N^2}{2 \cos\theta}
 \frac1{a^\flat  } \int_{z^\flat}^{2T+2|z|} f(s)  \, ds
\end{equation}
{\color{black} holds for any $T>0$}.
\end{lem}

\begin{proof} By Lemma \ref{xast2a},    one has
    $ (z+sa)^\flat\ge z^\flat\ge 2$ for each $s\ge0$.
By \eqref{F} and \eqref{sapa1}, we  then obtain
\begin{align*}
 \int_{0}^TF(z+sa)\,ds & \le \sum_{1\le i<j\le N} \int_{0}^T F_{ij}  ( z_i+sa_i,z_j+sa_j )  \, ds\\
& \le
 \sum_{1\le i<j\le N} \int_{0}^T f(|(z_i-z_j)+s(a_i-a_j)|)  \, ds.
   \end{align*}
We claim that,  for any $1\le i<j\le N$,   \begin{equation}\label{claimxx1}\int_{0}^T f(|(z_i-z_j)+s(a_i-a_j)|)  \, ds\le \frac1{\cos\theta}  \frac1{|a_i-a_j| } \int_{|z_i-z_j|}^{| z_i-z_j+T (a_i-a_j)|} f(t)  \, dt.
\end{equation}
Assume this claim holds for the moment. Note that  $$|z_i-z_j|\ge  z^\flat\ge 2,\,\, |a_i-a_j|\ge a^\flat>0$$  and
$$
| z_i-z_j+T(a_i-a_j)|\le 2|z|+ T|a_i-a_j|\le  2|z
|+ 2T.$$
One has
\begin{align*}
  \int_{0}^TF(z+sa)\,ds & \le   \frac{N^2}{2 \cos\theta}
 \frac1{a^\flat  } \int_{z^\flat}^{2T+2|z|} f(t)  \, dt,
   \end{align*}
that is,  \eqref{h1} holds.

Finally, we prove the above claim \eqref{claimxx1}.
For   any fixed
 $1\le i<j\le N$,  write $q=z_i-z_j$ and $e=a_i-a_j$,
and hence
$$\int_{0}^T f(|(z_i-s_j)+s(a_i-a_j)|)  \, dt= \int_{0}^{T} f(|q+se |)\,dt.$$
Write  $t(s) = |q+se | $ for $s\ge 0$.  Simce the assumption  $\angle (q,e)\le \theta<\pi/2$ implies
  $\langle q,e\rangle\ge |q||e|\cos\theta>0$, we have
$$t(s)=\sqrt{\langle q+se ,q+se \rangle}=\sqrt{|q|^2+2\langle q,e\rangle s+|e|^2s^2} \ge
|q| \ge 2, \quad \forall s\ge 0,$$
and $$  \frac{dt}{ds}=
  \frac{\langle   q+se ,e\rangle }{|q+se  |} \ge  \frac{|e||q|\cos\theta+|e|^2s} {|q + se |} \ge |e|\cos\theta\frac{ |q|+ {|e|s}  } {|q + se |} \ge |e|\cos\theta
  >0,  \ \forall s\ge 0 .$$
Thus $t=t(s):[0, \fz) \to [|q|,\fz)$ is a strictly increasing $C^1$ function.   Obviously,
 $$ \frac{ds} {dt}\le \frac1{|e| \cos\theta}. $$
And hence, by the change of variable, we have
$$    \int_{0}^{T} f(|q+se|)\,ds \le  \frac1{\cos\theta}  \frac1{|e|}\int_{|q |}^{|q+Te|}  f(t)      \,dt  $$
as desired.
\end{proof}

As a consequence of Lemma \ref{xast2} and Lemma 3.1, we have the following.
\begin{cor}\label{zabf}
Let $z\in \rr^{dN}$ with $z^\flat\ge 2$ and $a\in \mathbb S^{dN-1} $ with $a^\flat>0$.
If  $$|\frac z{|z|}-a|\le \frac {a^\flat}{20},$$
then for any $b\in \bs^{dN-1}$ with $|a-b|\le \frac{a^\flat}{20}$, we have
\begin{equation*}
\int_{0}^TF(z +sb)\,ds\le   N^2
 \frac2{a^\flat  } \int_{z^\flat}^{2T+2|z|} f(s)  \, ds,
 \quad \, \forall T >0.
\end{equation*}
\end{cor}

\begin{proof} Since $b\in \bs^{dN-1}$  and $|a-b|\le \frac{a^\flat}{20}$, by Lemma \ref{angle}, one has  $b^\flat> \frac 9{10} a^\flat $.
It follows that
$$|\frac z{|z|}-b|\le  |\frac z{|z|}-a| +|a-b|\le\frac {a^\flat}{10}\le \frac {b^\flat }9.$$
By Lemma \ref{angle}, it holds that
$$ \cos \angle (b_i-b_j,z_i-z_j)=\cos \angle (b_i-b_j,\frac{z_i}{|z|}-\frac{z_j}{|z|}) \ge 1-\frac 49\ge \frac12.$$
Thus by applying Lemma \ref{xast2} with $\theta=\arccos\frac12$ and $b^\flat> \frac 9{10} a^\flat $, we get the desired estimate.
\end{proof}

As a consequence of Corollary \ref{zabf} and Lemma \ref{findxast}, the following result holds:

\begin{cor}\label{zabf-2}
Let $x\in \rr^{dN}$  and $a\in \mathbb S^{dN-1} $ with $a^\flat>0$.
If  {\color{black} $s \ge 50\frac{1+|x|}{a^\flat}$},
then for any $b\in \bs^{dN-1}$ with $|a-b|\le \frac{a^\flat}{20}$, we have
\begin{equation}\label{h3}
\int_{0}^TF(x+sa +tb)\,dt\le   N^2
 \frac2{a^\flat  } \int_{(x+sa)^\flat}^{2S+2|x+sa|} f(t)  \, dt,
 \quad\forall T>0.
\end{equation}
\end{cor}
\begin{proof} Given  any $s \ge 50\frac{1+|x|}{a^\flat}$, let $z=x+sa $.
By Lemma \ref{findxast}, we know that $z^\flat\ge 2$ and
$|\frac{z}{|z|}-a|\le \frac{a^\flat}{20}$.
Then Corollary \ref{zabf} implies inequality \eqref{h3}.
\end{proof}

We now apply Corollary \ref{zabf-2} to prove Lemma \ref{uppm}.
\begin{proof} [Proof of Lemma \ref{uppm}]

By Lemma {\color{black}\ref{lem3-4}}, for any $T\ge 0$, one has  $$(x+sa +Tb)^\flat \ge (x+sa)^\flat\ge  sa^\flat-2|x| >48|x|\ge0. $$
It follows that $ x+sa+Tb  \in \Omega$. Thus $m_\lz (x,x+sa +Tb) <\fz$.

For any  $m_\lz$-geodesic
  $\gz:[0,\sz]\to\wz\Omega$  joining $x$ and    $x+sa+ Tb$, it holds that
$$l(\gz ) \le \frac{1}{\sqrt{2\lz}}m_{\lz}(x,x+sa+ Tb).$$
Note that
$$m_{\lz}(x,x+sa+ Tb) \le m_{\lz}(x,x+sa )+
  m_{\lz}(x+sa ,x+sa+ Tb).$$
By Lemma \ref{klem}, one has
$$  m_{\lz}(x+sa ,x+sa+ Tb)  \le  \sqrt{2\lz} T + \frac{1}{\sqrt{2\lz}} \int_{0}^TF(x+sa   +tb)\,dt .$$
By Corollary \ref{zabf-2}, it holds that
$$\int_{0}^TF(x+sa+tb)\,dt \le 2N^2\frac1{a^\flat}\int_{(x+sa)^\flat}^{2T+2| x+sa |} f(t)  \, dt ,$$
and hence
$$  m_{\lz}( x+sa , x+sa + Tb)  \le     \sqrt{2\lz} T+  \sqrt{\frac{2}{ \lz}}    N^2
 \frac1{a^\flat  } \int_{(x+sa)^\flat}^{2T+2| x+sa |} f(t)  \, dt.$$
We therefore obtain \eqref{eqk1}.
\end{proof}

\section{A key geometric  observation for $m_\lz$-geodesics}

To state our geometric observation, we first {\color{black}introduce some conventions}.

Given  $x\in \wz \Omega$ and
 $a\in \bs^{dN-1}$ with $a^\flat>0$, we always write
$$x^\ast:= x+   ( 50\frac{1+|x|}{  a^\flat })  a.$$
By  Lemma \ref{lem3-4}, one has
$$\mbox{  $(x^\ast +Tb)^\flat  \ge    2$, and hence {\color{black}$x^\ast+Tb\in \Omega$},
$ \forall T\ge0 $ and  $b\in\mathbb S^{dN-1}$ with $|a-b|\le \frac{a^\flat}{20}$. }$$


Given any $\lz>0$, we always  set
\begin{equation}\label{psi1}
\Psi(T)={\color{black}\Psi_{\lz,x,a}}(T):=\frac{1}{\sqrt{2\lz}}
m_{\lz}(x,x^\ast)+ \frac{1}{ \lz }  N^2
 \frac1{a^\flat  } \int_{(x^\ast)^\flat}^{2T+2|x^\ast|} f(t)  \, dt,  \quad\forall T\ge 0.
\end{equation}
Note that the  assumption \eqref{sapa1} guarantees
\begin{align}\label{size}
\sum_{j= 1}^\fz\sqrt{ 2^{-j}\Psi(2^{ j}) }<\fz.
\end{align}
Indeed, let $\xi \in \nn$ such that $ 2^{\xi-1} \le |x^\ast| < 2^{\xi}$.
{\color{black}By \eqref{psi1}, we have
\begin{align*}
\sum_{j= 1}^\fz\sqrt{ 2^{-j}\Psi(2^{ j}) }&\le \frac{1}{(2\lz)^{1/4}}\sum_{j\ge1}\sqrt{2^{-j}m_\lz(x,x^\ast)}+
\frac{N}{(\lz a^\flat)^{1/2}}\sum_{j\ge1}\sqrt{2^{-j}\int_{(x^\ast)^\flat}^{2^{j+1 }+2|x^\ast|  }f(t)\,dt}\\
&\le \frac{3}{(2\lz)^{1/4}}\sqrt{ m_\lz(x,x^\ast)}+ \frac{N}{(\lz a^\flat)^{1/2}} \sum_{j\ge1}\sqrt{2^{-j}\int_{1}^{2^{j+{ \xi} +2}  }f(t)\,dt}
\end{align*}
where in the last inequality we use the facts that $|x^\ast| < 2^{\xi} $ and $2^a+2^b \le 2^{a+b}$ whenever $a \ge 1$ and $b \ge 1$.}
Note that
{\color{black}
\begin{align*}
\sum_{j\ge 1}\sqrt{2^{-j}\int_1^{2^{j+{ \xi}+2}}f(t)\,dt}
&\le \sum_{j\ge 1}\sum_{k=0}^{j+{ \xi}+1} 2^{-j/2} \sqrt{ \int_{2^k}^{2^{k+1}}f(t)\,dt}
\le { C(x^\ast)}  \sum_{k\ge 0}\sqrt{2^{-k}\int_{2^k}^{2^{k+1}}f(t)\,dt}
\end{align*}}
for some positive constant $C(x^\ast)$. {\color{black} Thanks to \eqref{sapa1}, the right-hand side of the last inequality is
finite.} Thus \eqref{size} holds.

We further write {\color{black}
$$\wz \Psi(t): =\sum_{j\ge \lfloor \log_2 t \rfloor +1}\sqrt{2^{-j}\Psi(2^j)}<\fz,
\quad\forall \, t \ge 1 .
$$}
Then $\wz\Psi(t)$ is decreasing  to $0$ as $t\to\fz$. {Moreover,
\begin{equation}\label{rel}
  \sqrt{ t^{-1}\Psi(t)}\le \wz \Psi(t), \quad\forall t\ge 1.
\end{equation} \color{black}
To  see this, letting $2^k \le t < 2^{k+1}$ for some $k \in \nn$, we have
{\color{black}
$$  \wz \Psi(t) =\sum_{j\ge k+1} \sqrt{\frac{\Psi(2^j)}{2^{j}}} \ge \sqrt{\Psi(2^{k+1})} \sum_{j\ge k+1} \sqrt{\frac{1}{2^{j}}} \ge \sqrt{\Psi(2^{k+1})} \frac{\sqrt{2^{-(k+1)}}}{1-\sqrt{2}^{-1}} \ge \sqrt{  \frac{\Psi(2^{k+1})}{2^{k}}} \ge \sqrt{ \frac{\Psi(t)}{t}},$$}
where in the first and last ``$\ge$'' we use the fact that $\Psi(t)$ is increasing  with respect to $t$, and in the second last ``$\ge$'' we use the fact $(1-\sqrt{2}^{-1})^{-1} \ge \sqrt{2}$.

From \eqref{rel},
it follows that }
\begin{align}\label{n0}
n_0:=\min\left\{n: n\ge 20+\log_2 |x-x^\ast|, \, \wz \Psi(2^n)=\sum_{j= n}^\fz\sqrt{ 2^{-j}\Psi(2^{ j}) }\le 2^{-10}a^\flat\right\}<\fz.
\end{align}
Obviously, one has
\begin{align}\label{size2}\Psi(T)\le 2^{-20}  (a^\flat)^2 T, \quad\mbox{for all $T\ge 2^{n_0}$}.
\end{align}

We remark that with the aid of $x^\ast$ and $\Psi$, we  restate Lemma \ref{uppm}  with $s=     50\frac{1+|x|}{  a^\flat }   $ as below.

\noindent {\bf Lemma 4.1$^\ast$}. {\it Let $b\in \mathbb S^{dN-1}$ with $|a-b|\le \frac{ a^\flat}{20}$.
For any   $T>0$, and for any
   $m_\lz$-geodesic $\gz$ with canonical parameter joining $x$ and   $x^\ast+ Tb$, we have
\begin{equation}\label{eqk2}
  l(\gz ) \le \frac1{\sqrt{2\lz}}m_\lz(x,x^\ast+Tb) \le  T+ \Psi(T).
\end{equation}
}

We are ready to state our geometric observation as below, which plays a key role in the proof of Theorem 1.4.

\begin{lem}  \label{tan1}
Let $\gz\in \mathcal{AC}(x,x^\ast+2^na;[0,\sz];\wz\Omega)$ be an  $m_\lz$-geodesic with canonical parameter, where
  $n>n_0$.
Let $b\in \bs^{dN-1}$ with $|a-b|\le \frac{a^\flat}{10}$.
Suppose  that $ x^\ast +Sb= \gz ( \tau) $  for some  $S\ge 2^{n_0}$ and $ \tau\in (0,\sz]$.
Then   \begin{align} \label{pw1}
 |\gz (  \tau_{1/2}) -(x^\ast + \frac S2 b)| &\le 2 \sqrt{  S \Psi(S)} ,
\end{align}
where   $$ \tau_{1/2} :=\max\{0<t< \tau: |x^\ast-\gz (t)|=\frac S2\};$$
 and moreover,
\begin{equation}\label{zq2}
 d_E(\gz(t),  x^\ast+\rr_+b )
\le  4\sqrt {S \Psi(S)}, \quad\forall t\in[ \tau_{1/2},\tau].
\end{equation}
\end{lem}

We first prove \eqref{pw1} in Lemma \ref{tan1}.

\begin{proof} [Proof of \eqref{pw1} in Lemma \ref{tan1}]
 Let $$\hat  \tau_{1/2}:= \max\{0<t< \tau: |x^\ast-\gz (t)|=|\gz  (t)-(x^\ast+bS)|\}.$$
Observing   $|\gz(\hat\tau_{1/2})-x^\ast|\ge \frac S2$, we have $\tau_{1/2}\le \hat\tau_{1/2}$ (See Figure \ref{f1} for an illustration).
By \eqref{size2},
\begin{equation}\label{psiss}
\Psi(S)\le 2^{-20}  (a^\flat)^2 S \le 2^{-18} S.
\end{equation}
To get \eqref{pw1}, it suffices to prove that
{\color{black}\begin{align} \label{pw1-1}
|\gz (\hat \tau_{1/2}) -(x^\ast +b\frac S2)|
&\le (1+ 2^{-19}) \sqrt{ S\Psi(S)  }
\end{align}}
and
\begin{align}\label{pq1}
|\gz (\hat \tau_{1/2})-\gz ( \tau_{1/2})|&\le 2\Psi(S) \le 2^{-8} \sqrt{ S\Psi(S)}.
\end{align}

{\color{black}Now}, we prove the above two inequalities: \eqref{pw1-1} and \eqref{pq1}. As in Figure \ref{f1}, we write
$$\mbox{$y=x^\ast + Sb$, \,  \, $w =x^\ast +\frac S2 b $,\, \,
  $p=\gz (\hat  \tau_{1/2})$ \, and \, $q= \gz (   \tau_{1/2})$.}$$
Note that  $w$ is  the middle point of the straight line segment $[x^\ast, y]$. One has $$|x^\ast-p|=|p-y|>\frac S2
 \ \mbox{
and} \ |x^\ast-w|=|w-y|=\frac12|x^\ast-y|=\frac S2.$$  Obviously, $$\la p-w,x^\ast-y\ra=\la  p-w,b\rangle =0.$$
\vspace*{7pt}
\begin{figure}[h]
\centering
\includegraphics[width=12cm]{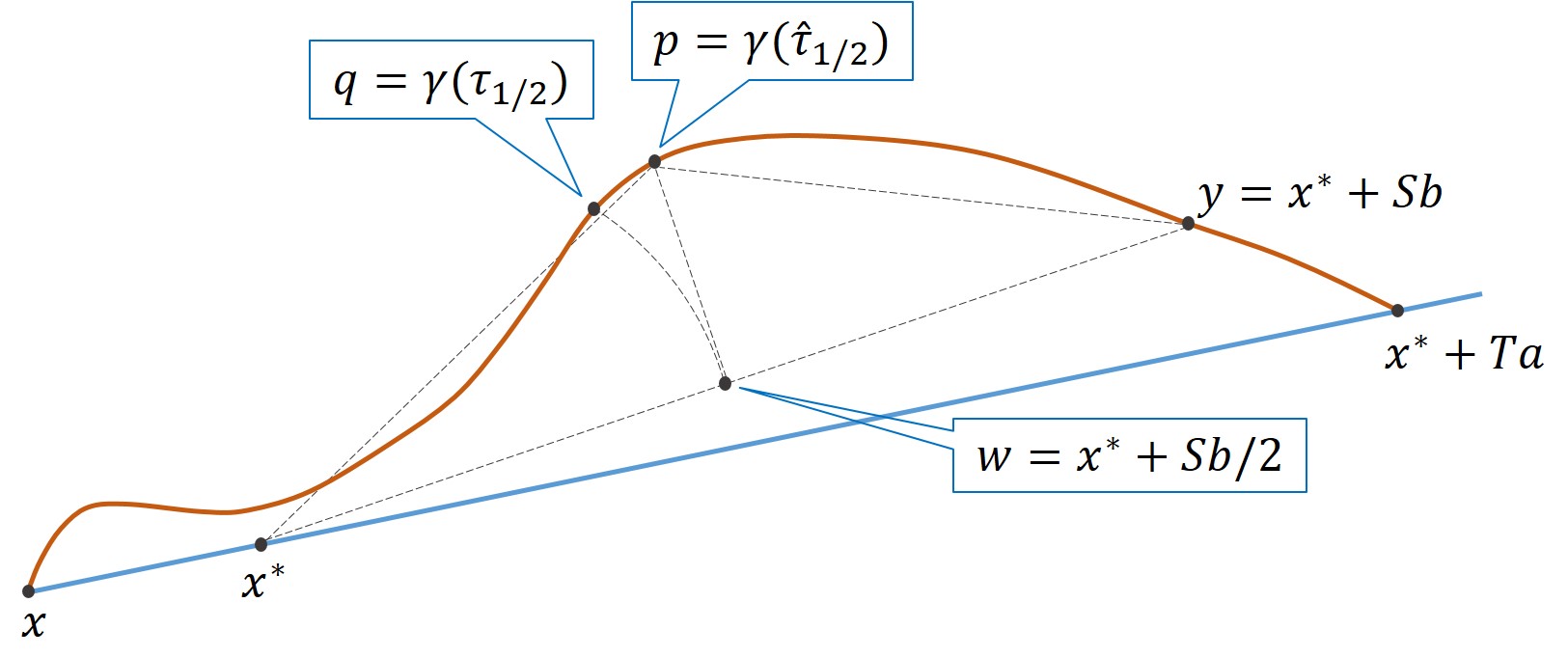}
\caption  {The distance between $p$ and $w$.}
\label{f1}
\end{figure}

{\it Proof of \eqref{pw1-1}.}
Note that
\begin{align*}|x^\ast- p|&=\frac12 \left[
 |x^\ast- p|+| p-y|\right]  \le \frac12 \left[
 |x- p|+| p-y|+ |x-x^\ast| \right].
\end{align*}
Since  $|x- p|\le l(\gz |_{[0,\hat  \tau_{1/2}]} )$ and $| p-y|\le l(\gz |_{[\hat  \tau_{1/2} , \tau] })$, we obtain
\begin{align*}
|x^\ast- p| & \le \frac12 \left[ l(\gz |_{[0,\hat  \tau_{1/2}] })+l(\gz |_{[\hat  \tau_{1/2} , \tau]} )+|x-x^\ast|\right]
 =\frac12 l(\gz |_{[0, \tau]} )+ \frac12 |x-x^\ast|.
\end{align*}
By \eqref{eqk2}, $$l(\gz|_{ [0, \tau]} )\le  S+\Psi(S),$$
we therefore obtain
\begin{align*}
|x^\ast- p| &\le \frac12 S+\frac12\Psi(S)+ \frac12|x-x^\ast|.
\end{align*}
Note that $S \ge 2^{n_0}$, {\color{black}$\Psi(S) \ge \Psi(2^{n_0})$ and $\Psi(S)\le 2^{-18} S$}.
By the definition of $n_0$ in \eqref{n0} and Lemma \ref{klem}, we know that
\[ |x-x^\ast| \le 2^{n_0-20} \le 2^{-20} S\quad\mbox{and}  \quad |x-x^\ast|\le \frac{1}{\sqrt{2\lz}}m_\lz(x,x^\ast) < \Psi(S). \]
 Thus,
{\color{black}\begin{align*}
|p -w | &=\sqrt{ |x^\ast- p |^2-|x^\ast-w|^2 }  \\
& \le \sqrt {[ \frac S2+\frac12\Psi(S)+ \frac12|x-x^\ast| ]^2-[\frac S2]^2  } \\
&  =  \sqrt{ \frac12\Psi(S) [S+ \frac12\Psi(S)] + \frac12|x-x^\ast| [\frac12|x-x^\ast| +S +  \Psi(S)] } \\
& \le \sqrt{S \Psi(S) }\left( 1 + 2^{-19} + 2^{-20}  +2^{-22} \right)^{1/2} \\
& \le  \sqrt{S \Psi(S) }(1+  2^{-18})^{1/2} \\
& \le  \sqrt{S \Psi(S) }(1+ 2^{-19})
\end{align*}}
where in the last inequality we applied $(1+c)^{1/2} \le 1 + \frac{c}{2}$ for any $c\ge 0$.
Therefore, \eqref{pw1-1} holds.

\medskip

 {\it Proof of \eqref{pq1}.}
Recalling $p=\gz  (\hat  \tau_{1/2})$  and $\la p-w,w-y\rangle=0$ one has
$$l(\gz |_{[\hat  \tau_{1/2},\tau ]})\ge |p-y | \ge |w -y|=\frac S2,$$
and
$$l(\gz |_{[0,\hat \tau_{1/2}]})\ge |x-p|\ge |x^\ast-p|- |x^\ast -x|\ge \frac S2-|x^\ast -x|.$$
From this and \eqref{eqk2},   it follows that
\begin{align*}l(\gz |_{[\hat \tau_{1/2},\tau]})&= l(\gz |_{[0,\tau]})-l(\gz |_{[0,\hat \tau_{1/2}]})
 \le  S+\Psi(S)-\frac S2  +|x^\ast -x|= \frac S2+\Psi(S)+|x^\ast -x|.
\end{align*}

Moreover, by recalling $|x^\ast -q|=S/2$ and Lemma \ref{klem}, we have
  $$ l(\gz |_{[0, \tau_{1/2}]}) \ge|x-q| \ge |q-x^\ast|-|x^\ast-x|\ge
\frac S2-\frac{1}{\sqrt{2\lz}}m_\lz(x,x^\ast)\ge \frac S2-\Psi(S). $$
From this and \eqref{eqk2}, one deduces that
  \begin{equation}\label{tau121}
l(\gz |_{[ \tau_{1/2},\tau]})=l(\gz |_{[0,\tau]})-l(\gz |_{[0, \tau_{1/2}]})  \le
S+\Psi(S)-\frac S2+\Psi(S)\le \frac S2+2\Psi(S).
  \end{equation}

\vspace*{7pt}
\begin{figure}[h]
\centering
\includegraphics[width=12cm]{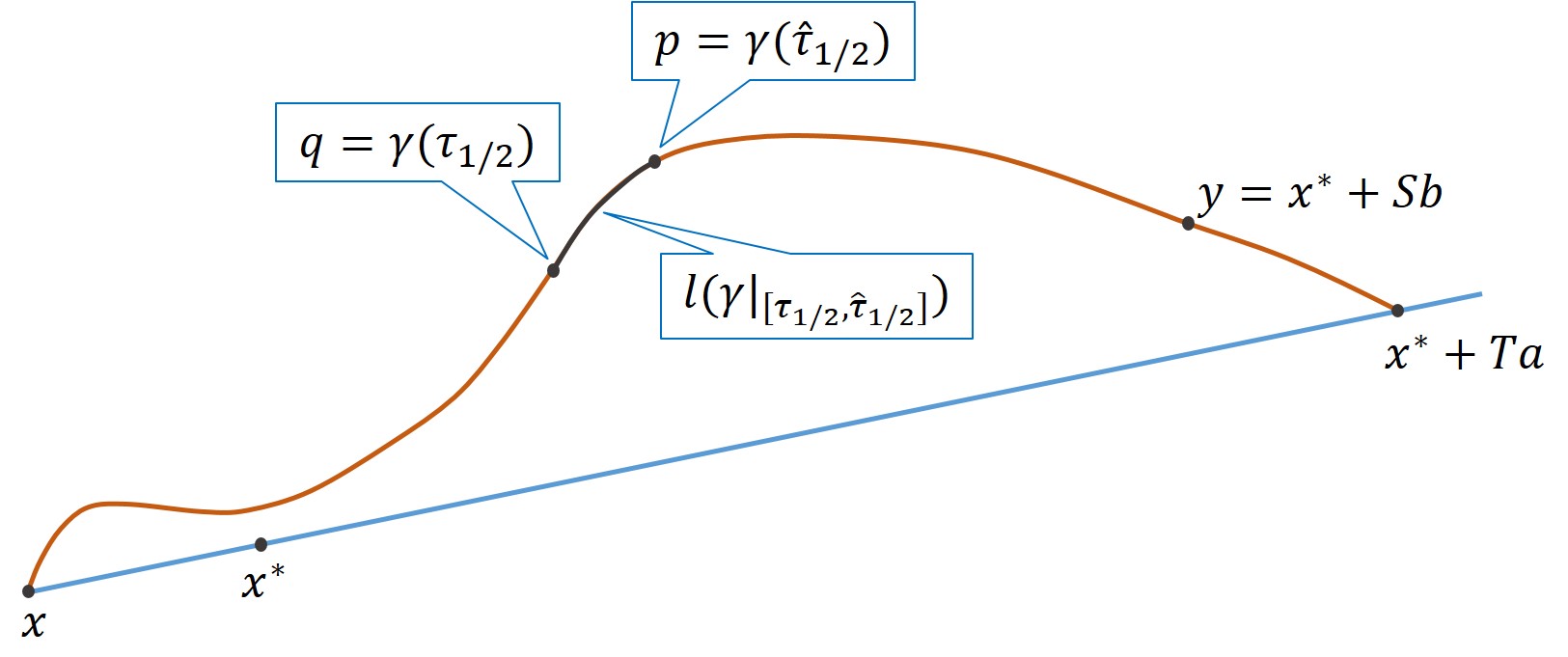}
\caption  {The length of $l(\gz |_{[ \tau_{1/2},\hat \tau_{1/2} ]})$.}
\label{f2}
\end{figure}

Since  $ \tau_{1/2}\le \hat  \tau_{1/2}$,   one has  $$ |p-q|\le  l(\gz |_{[ \tau_{1/2},\hat \tau_{1/2} ]})
 =l(\gz |_{[ \tau_{1/2},\tau]})-  l(\gz |_{[\hat  \tau_{1/2},\tau ]})\le \frac S2+2\Psi(S)- \frac S2\le 2\Psi(S) \le 2^{-8} \sqrt{ S\Psi(S)}.  $$
Thus,   $$ |p-q|\le  2 \Psi(S) \le 2^{-8} \sqrt{ S\Psi(S)},$$
which shows \eqref{pq1}.
\end{proof}

Next we prove \eqref{zq2} in Lemma \ref{tan1}.

\begin{proof}[Proof of \eqref{zq2} in Lemma \ref{tan1}]

Recall that $y=\gz ( \tau )$ and  $q  =\gz ( \tau_{1/2})$. For any $t\in (\tau_{1/2},\tau)$, the definition of $\tau_{1/2}$ implies that $|\gz(t)-x^\ast|\ge \frac S2${\color{black}, moreover
there is  a unique $z \in [q,y]$}  such that
$$d_E(\gz(t), {\color{black}[q,y]}) =|\gz(t)-z|.$$
It then suffices to prove
\begin{equation}\label{zq1}
 |\gz(t)-z| \le  2\sqrt {S \Psi(S)}.
\end{equation}
Indeed,
considering
 $y=\gz ( \tau)$, inequalities \eqref{pw1-1} and \eqref{pq1} imply that
$$d_E(z, x^\ast+\rr_+b)\le d_E(z,[x^\ast,y])\le d_E(q,[x^\ast,y])\le   |q-p|+|p-w|\le 2\sqrt{S\Psi(S)}.$$
Thus
$$d_E(\gz(t), x^\ast+\rr_+b) \le |\gz(t)-z|+d_E(z,x^\ast+\rr_+b)   \le  4\sqrt {S \Psi(S)},$$
which {\color{black}
shows \eqref{zq2}. See Figure \ref{f3}} for illustration.
\vspace*{7pt}
\begin{figure}[h]
\centering
\includegraphics[width=12cm]{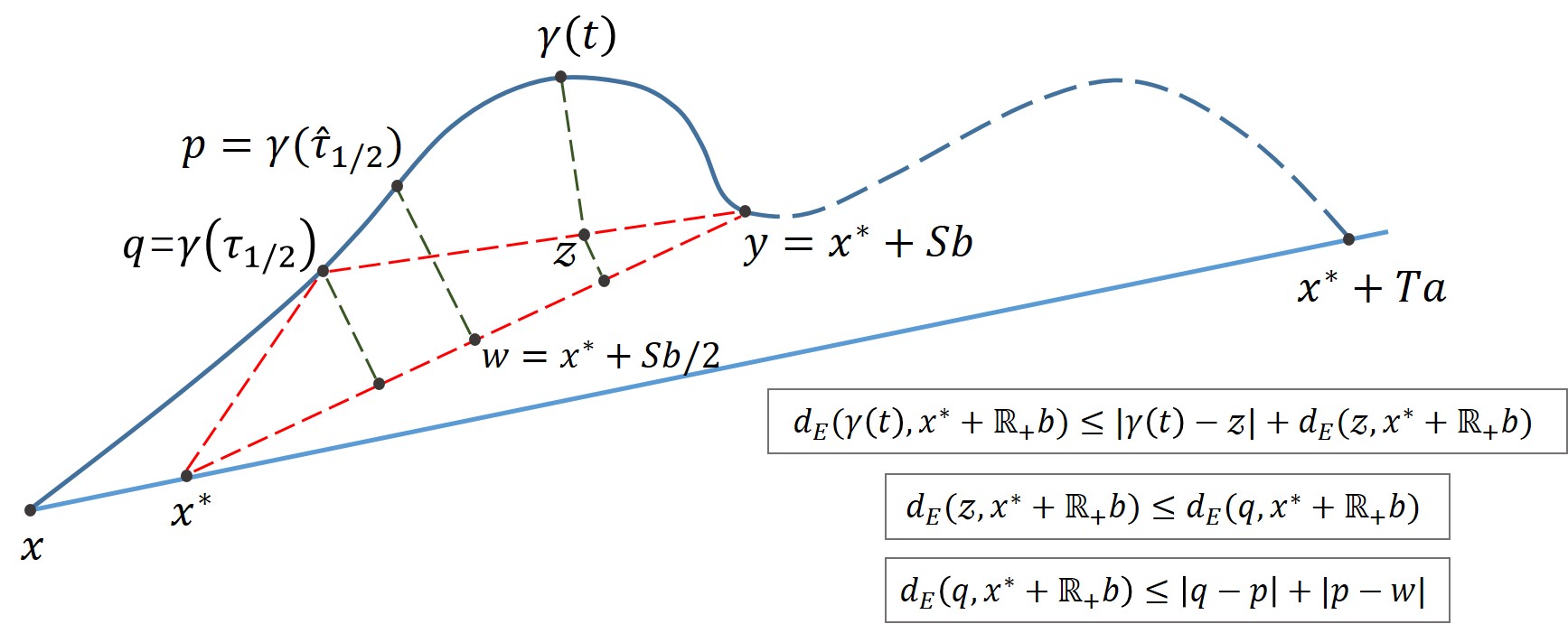}
\caption  {An illustration of showing \eqref{zq2}.}
\label{f3}
\end{figure}

Finally, we prove  \eqref{zq1} by considering three cases: Case $z \in (q,y)=(\gz ( \tau_{1/2}),\gz ( \tau))  $,   Case $z= y=\gz ( \tau )$ and
Case $z= q=\gz ( \tau_{1/2})$.  Recall that $(q,y):=\{sq+(1-s)y: s\in(0,1)\}$.
Moreover, since $\gz|_{[\tau_{1/2},t]}$ {\color{black} joins $q$ to $\gz(t)$} and also $\gz|_{t,\tau]}$ joins {\color{black}$ \gz(t)$ to $y$}, we always have
 \begin{equation}\label{ltau12tau}|q-\gz(t)|+|\gz(t)-y|\le l(\gz|_{[\tau_{1/2},t]})+ l(\gz|_{[t,\tau]})= l(\gz|_{[\tau_{1/2},\tau]}) .\end{equation}

\medskip
{\it   Case $z \in (q,y)  $}.
The proof in this case is illustrated by  Figure \ref{f4}.
\vspace*{7pt}
\begin{figure}[h]
\centering
\includegraphics[width=12cm]{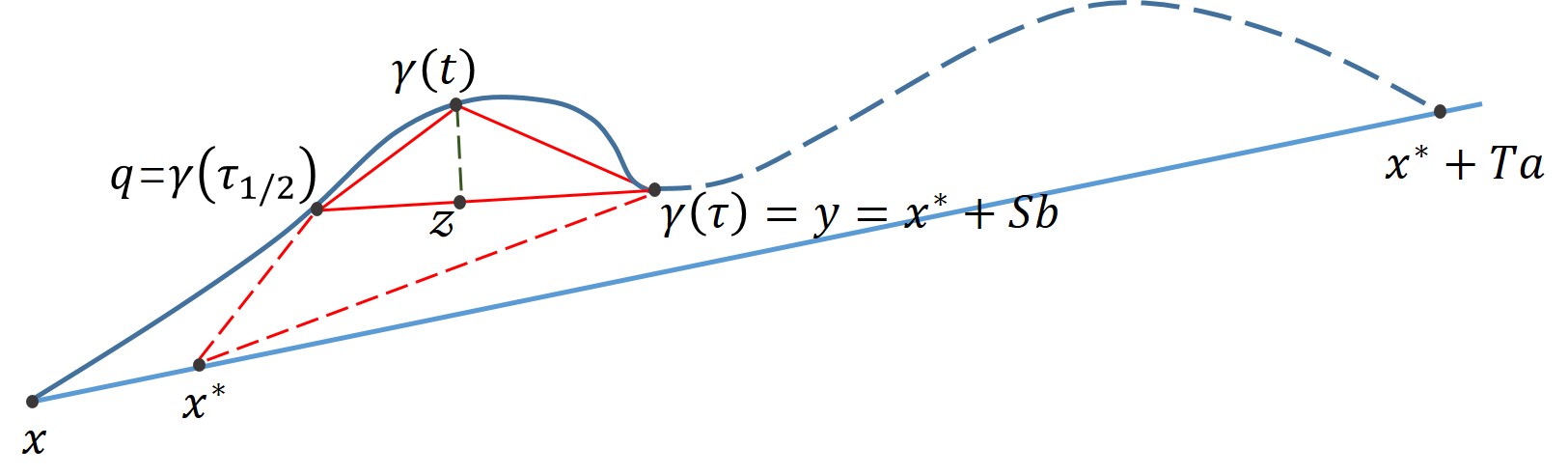}
\caption  {An illustration of Case  $z\in (q,y)  $.}
\label{f4}
\end{figure}

In this case, we first observe that
\begin{equation}\label{perp}\langle \gz(t)-z, q-y\rangle=0.
\end{equation}
Geometrically, \eqref{perp} is obvious. A standard proof to \eqref{perp} goes as follows. If $ |\dz|<\dz_0$ for some $\dz_0>0$,
 one has  $z +\dz(q-y)\in (q,y).$
By the definition of $z$,  the function $$h(\dz):=|\gz(t)-z- \dz(q-y) |^2$$ reaches
minimal in $(-\dz_0,\dz_0)$ at   $\dz=0$.   Thus $0=h'(0)=-2\langle \gz(t)-z, q-y\rangle$, which gives
 \eqref{perp}.

%

By \eqref{perp}, we have
\begin{align*}
|\gz(t)-z|^2&=\frac12\left[|q-\gz(t)|^2-|q-z| ^2+|\gz(t)-y|^2-|y-z| ^2\right]\\
&= \frac12[(|q-\gz(t)|+|\gz(t)-y|)^2-2|q-\gz(t)| |y-\gz(t)|\\
&\quad\quad\quad\quad\quad -(|q-z|+|y-z| ) ^2 +2|q-z||y-z|  ].
\end{align*}

Note that $z\in(q,y)$ implies $$ |q-z|+|y-z| =|q-y|.$$
By $\langle \gz(t)-z,q-y\rangle=0$, one has
$$|q-z|\le |q-\gz(t)|,\quad  |y-z|  \le |y-\gz(t)|, $$
and hence  $$|q-\gz(t)| |y-\gz(t)|\ge |q-z||y-z|.$$
These together with \eqref{ltau12tau} yield
\begin{equation}\label{gztz}
|\gz(t)-z|^2 \le \frac12\left[l(\gz|_{[\tau_{1/2},\tau]})^2-|q-y|^2\right].
\end{equation}
 Since $|q-x^\ast|=\frac S2$,   one has
\begin{equation}\label{qys2}
|q-y|\ge |x^\ast-y|-|x^\ast-q|= \frac S2.
\end{equation}
By \eqref{tau121}, \eqref{gztz} and \eqref{qys2}, we further obtain
\begin{align*}
|\gz(t)-z|^2
&\le \frac12\left[(\frac S2+2 \Psi(S))^2-(\frac S2)^2\right] =   S\Psi(S)+ 2\Psi(S)^2.
\end{align*}
Since    $\Psi(S)< 2^{-18} S$ by \eqref{psiss}, we have
$ |\gz(t)-z| < 2 \sqrt{S\Psi(S)} $ as desired.
%


{\it Case  $z=y$.}  The proof in this case is illustrated by Figure \ref{f5} below.
\vspace*{7pt}
\begin{figure}[h]
\centering
\includegraphics[width=12cm]{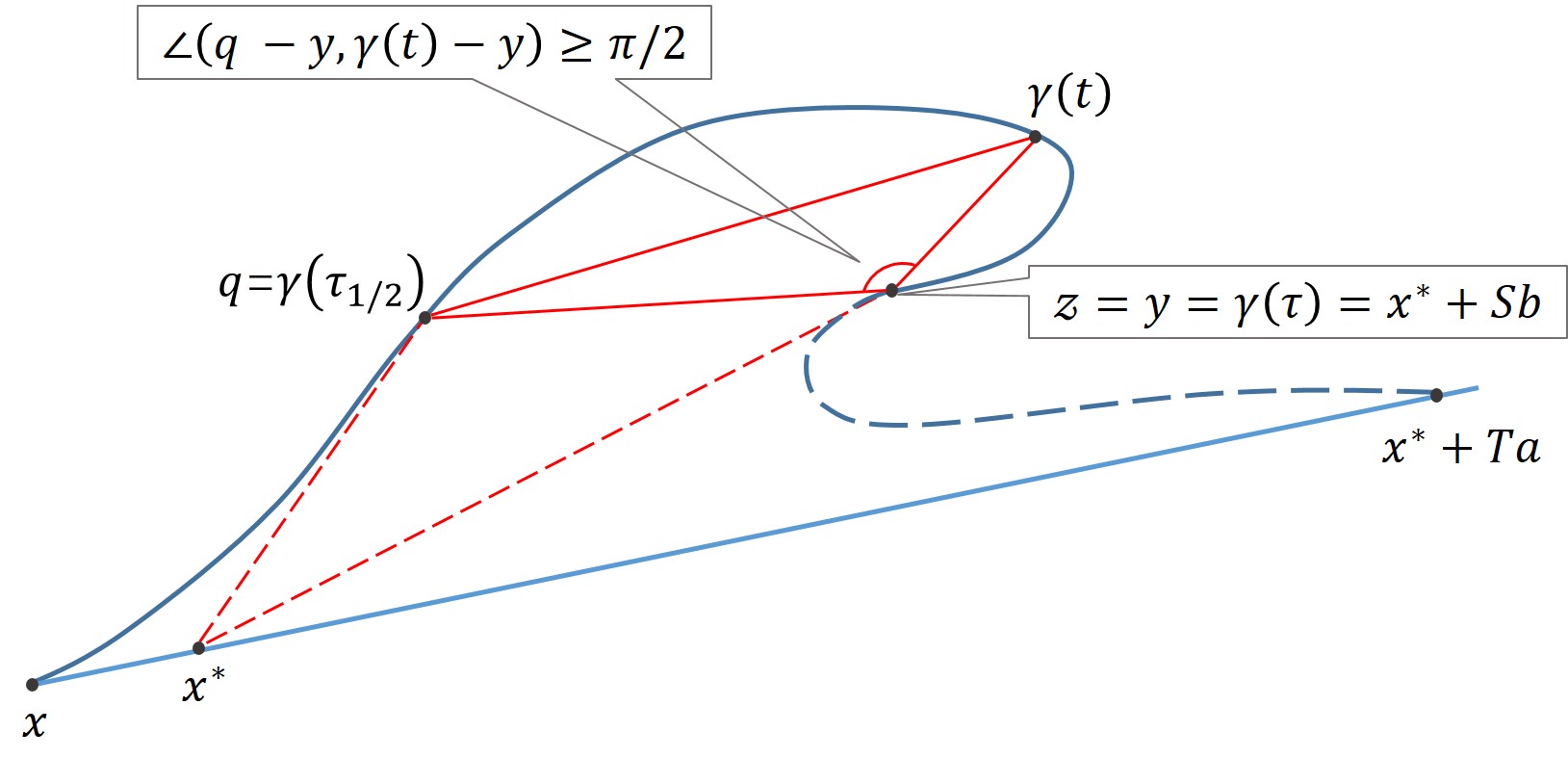}
\caption  {An illustration of Case  $z\notin [\gz ( \tau_{1/2}),\gz ( \tau)]  $.}
\label{f5}
\end{figure}

 In this case,  note that
\begin{equation}\label{anglegepi}\angle (\gz(t)-y, q-y)\ge \frac \pi2,
\end{equation}
Geometrically, \eqref{anglegepi} is obvious. We give a proof. If $  0\le \dz\le 1 $,
 one has $$z +\dz(q-y)= y +\dz(q-y)\in [q,y].$$
By the definition of $z$,  the function $$h(\dz):=|\gz(t)-z- \dz(q-y)|^2$$
reaches minimal in $[0,1]$ at   $\dz=0$.   Thus $$ 0\le \lim_{\dz\to 0+}\frac1{\dz}[h(\dz)-h(0)]=-2\langle\gz(t)-y,q-y\rangle  $$
which gives \eqref{anglegepi}  as desired.

Considering the triangle $\triangle  (\gz(t), \gz(\tau),\gz(\tau_{1/2}) )$. \eqref{qys2} and \eqref{anglegepi} imply that
\[ |\gz(t)-q| \ge |q-y|\ge \frac S2.\]
Applying \eqref{ltau12tau} and \eqref{tau121}, by $z=y$ we therefore obtain
 $$ |\gz(t)-z|= \left[|\gz(t)-q|+ |\gz(t)-y|\right]-|\gz(t)-q| \le  \frac S2+2\Psi(S)- \frac S2=
2\Psi(S)<2\sqrt{S\Psi(S)}.$$

{\it Case $z=q$.}   The proof in this case is much similar to   Case $z=y$. We omit the details.

{\color{black}This ends the proof}.
\end{proof}

\section{ Estimates of Euclidean length and angle  of $m_\lz$-geodesics  }

Given  any $x\in \wz \Omega$, any
 $a\in \bs^{dN-1}$ with $a^\flat>0$, and any
$\lz>0$, let   $x^\ast$, $\Psi$, $\wz\Psi$  and $n_0$ be as in Section 5.

%
%

 \begin{lem}  \label{unif} Let $\gz\in \mathcal{AC}(x,x^\ast+2^na;[0,\sz];\wz\Omega)$ be an  $m_\lz$-geodesic with canonical parameter, where
  $n>n_0$. Set   $$\sz_{n_0}:=\max\{s\in[0,\sz]: |\gz(s)-x^\ast|=2^{n_0}\}.$$
For $t\in [\sz_{n_0},\sz]$, we have
\begin{align}\label{distanceest}
| \frac{\gz(t)-x^\ast}{|\gz(t)-x^\ast|}-a | \le 2^4 \wz\Psi( |\gz(t)-x^\ast|)
 \end{align}
and
\begin{align}\label{sizeest}l(\gz|_{[0,t]})\le \frac1{\sqrt{2\lz}}m_{\lz}(x,\gz(t))\le |\gz(t)-x^\ast|(1+\wz\Psi^2(|\gz(t)-x^\ast|)).
\end{align}
 In particular,  \begin{align}\label{flatgz}
\mbox{
$\gz(t) \in \Omega$  with
$(\gz(t))^\flat\ge (x^\ast)^\flat>0 $ for all $t\in[\sz_{n_0},\sz]$.}\end{align}

%
%
\end{lem}

\begin{proof}
For $ n_0\le i\le n $, set
$$ \sz_i:=\max\{t\in(0,\sz): \,|x^\ast-\gz (t)| = 2^{ i}  \}.$$
Write
 $$\mbox{$a^{(n)} =a$, and $a^{(i)} :=2^{-i}(\gz ({ \sz_i })-x^\ast)$ for $n_0 \le i\le n-1$.} $$
Obviously,  $$a^{(i)}= \frac{\gz ({ \sz_i })-x^\ast}{|\gz ({ \sz_i })-x^\ast|}\in \bs^{dN-1}.$$

{\it Step 1.} We  claim that
\begin{equation}\label{exx}|a^{(i )}-a^{(i+1)}|\le  4\sqrt{     2^{-(i+1)} \Psi(2^{i+1} )  }, \quad
\mbox{for any $n_0 \le i\le n-1$}.
\end{equation}
We prove this  by induction; see   Figure 6 for an illustration.
 \vspace*{7pt}
\begin{figure}[h]
\centering
\includegraphics[width=12cm]{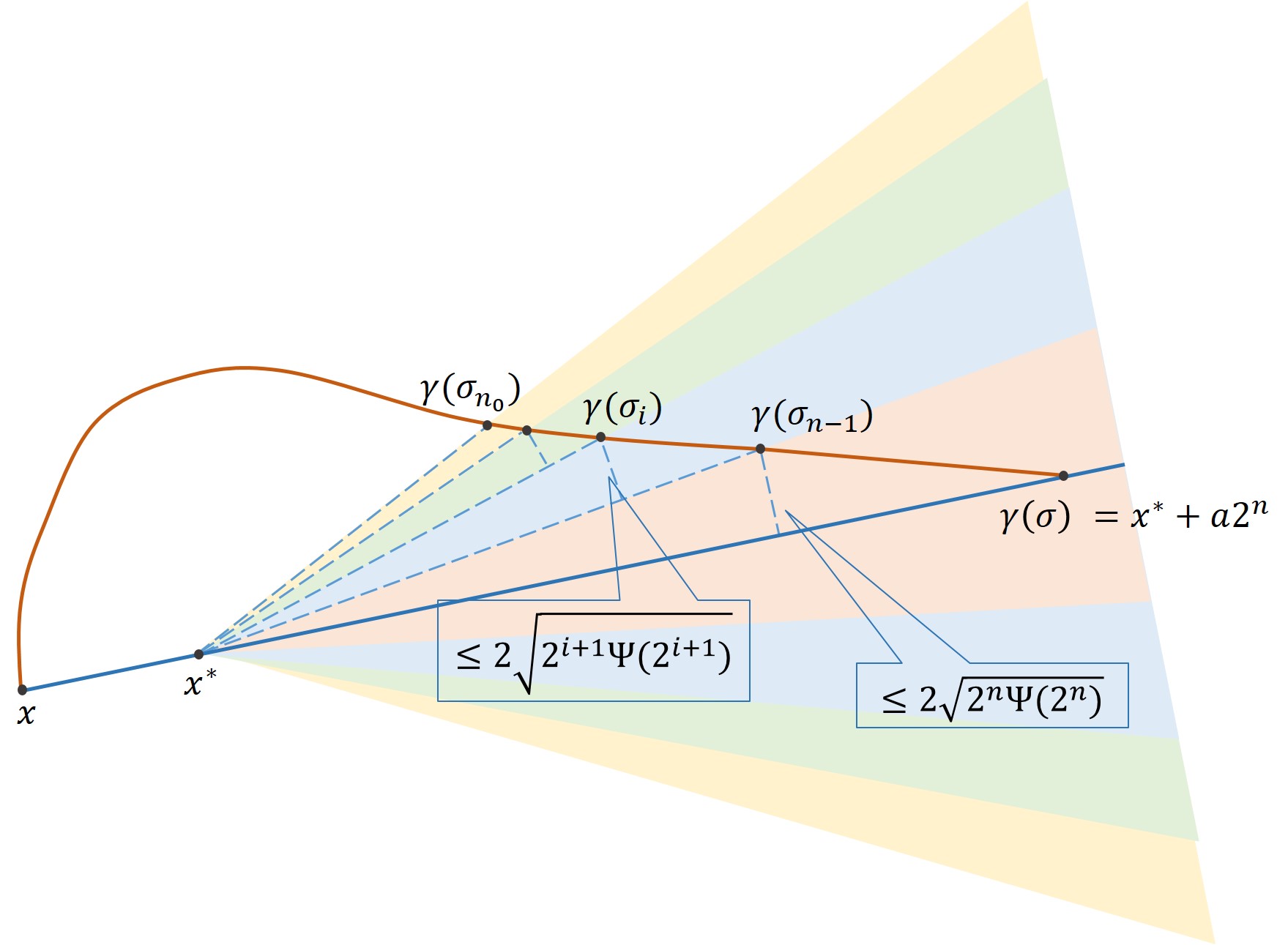}
\caption  {An illustration of the induction process.}
\label{f6}
\end{figure}

 \noindent
  First, we see that \eqref{exx} holds for $i=n-1$. Indeed, applying  Lemma \ref{tan1} with
 $b=a=a^{(n)}$, $\tau=\sz$ and $\tau_{1/2}=\sz_{n-1}$, we know that
$$| \gz ( \sz_{n-1})-(x^\ast+ 2^{n-1}   a)|\le 2\sqrt{ 2^{n} \Psi(2^n) }.$$
Hence dividing both sides by $2^{n-1}$ in the above inequality, we have
$$|a^{(n-1)}-a^{(n)}|=| \frac{\gz ( \sz_{n-1})-x^\ast}{2^{n-1}} -a|  \le 4\sqrt{ 2^{-n}\Psi(2^n)   }\le 2^{-8} a^\flat,$$
 where in the last inequality we use \eqref{size2}. Then by Lemma \ref{angle}, we know $a^{(n-1)} \in \Omega$.

Next, given any $j\ge n_0$, assume that \eqref{exx} holds  for all $j+1\le i\le  n-1$. By the definition of $n_0$ in \eqref{n0}, it holds that $\wz \Psi(2^{n_0}) \le 2^{-10} a^\flat.$ This implies $$|a^{ (j+1)}-a^{(n)}| \le  \sum_{i=j+1}^{n-1} |a^{(i )}-a^{(i+1)
}|\le \sum_{i=j+2}^\fz 4\sqrt{   2^{-i}  \Psi(2^{i })  }=  4\wz \Psi(2^{j+2}) \le 4 \wz \Psi(2^{n_0})   \le 2^{-8}  a^\flat.$$
  Here, again, in the last inequality we use \eqref{size2} and    $a^{(j)} \in \Omega$. By this,  applying Lemma \ref{tan1} with  $b=  a^{( j+1)}$, $\tau=\sz_{j+1}$ and $\tau_{1/2}=\sz_j$,  we have
$$| \gz ( \sz_{ j  })-(x^\ast+ 2^{j }  a^{(j+1)})|\le 2\sqrt{  2^{j+1}  \Psi(2^{ j+1} ) }$$
Since
 $\gz ( \sz_{j })=x^\ast+   2^{ j }  a^{( j )} $, we have
\begin{equation}\label{exxj}|a^{(j )}-a^{(j+1)}|\le  4\sqrt{     2^{-(j+1)}  \Psi(2^{ j+1} )  }.
\end{equation}
 that is, \eqref{exx} holds  for $i=j$.

Thus  \eqref{exx} holds  for $n_0\le i\le n-1$ as desired.

\medskip
{\it Step 2.} We show \eqref{distanceest}, that is,
$$| \frac{\gz(t)-x^\ast}{|\gz(t)-x^\ast|}-a |\le   16 \wz\Psi( |\gz(t)-x^\ast|), \quad {\color{black} \text{for $t\in [\sz_{n_0},\sz]$}}.$$

By Step 1, {\color{black}if $n_0 \le i\le n-1$,} we have
$$|a^{(i )}-a^{(n)}| \le  \sum_{j=i+1}^{n } |a^{(j-1)}-a^{(j)}| \le \sum_{j=i}^\fz 4\sqrt{   2^{-j}  \Psi(2^{j })  }=  4\wz \Psi(2^{i}).
$$
That is,
$$| \frac{\gz(\sz_i)-x^\ast}{|\gz(\sz_i)-x^\ast|}-a |  \le 4 \wz \Psi(|\gz(\sz_i)-x^\ast|). $$

 For any  $t\in [\sz_{i },\sz_{i+1}]$,
by the definition of $  \sz_i $ and $\sz_{i+1}$, we know $$ |\gz(t)-x^\ast|\ge |\gz( \sz_i)-x^\ast| = 2^i.$$
Applying Lemma \ref{tan1} with $b= a^{(i+1)}$ in \eqref{zq2}, $\tau=\sz_{i+1}$ and $\tau_{1/2}=\sz_i$, by the increasing property of $\Psi$ we have
\[d_E(\gz(t),  x^\ast+\rr_+a^{(i +1)} ) \le    4\sqrt{ 2^{i+1 }\Psi(2^{i+1})}\le 4 \sqrt{  2 |\gz(t)-x^\ast| \Psi(2 |\gz(t)-x^\ast|)}.\]
Then $$\sin \angle(\gz(t)-x^\ast, a^{(i+1)})= \frac {d_E(\gz(t),  x^\ast+\rr_+a^{(i +1)} ) }{|\gz(t)-x^\ast|}\le \frac{ 4\sqrt{ 2^{i+1}\Psi(2^{i+1})}}{2^i} .$$
By \eqref{psiss}, we know that $2^{-(i+1)} \Psi(2^{i+1}) \le  2^{-18}$ and hence
$$\sin \angle(\gz(t)-x^\ast, a^{(i+1)}) \le  2^{-6}.$$
This also gives
 $$\cos\angle(\gz(t)-x^\ast, a^{(i+1)}) \ge  \sqrt{1-2^{-12}}\ge \frac 9{10}.$$
It follows that
{\color{black}
\begin{align*}| \gz(t)-(x^\ast+|\gz(t)-x^\ast|a^{(i+1)}) | &= d_E(\gz(t),  x^\ast+\rr_+a^{(i +1)} )  \frac{1}{\cos (\frac12  \angle(\gz(t)-x^\ast, a^{(i+1)}))} \\
& \le d_E(\gz(t),  x^\ast+\rr_+a^{(i +1)} )  \frac{1}{\cos \angle(\gz(t)-x^\ast, a^{(i+1)})} \\
& \le \frac{10}{9} d_E(\gz(t),  x^\ast+\rr_+a^{(i +1)} ) .
\end{align*}}

We  conclude that
\begin{align*}
| \frac{\gz(t)-x^\ast}{|\gz(t)-x^\ast|}-a^{(i+1)}|&=
 \frac { | \gz(t)-(x^\ast+|\gz(t)-x^\ast|a^{(i+1)}) |}{|\gz(t)-x^\ast|}\\
&\le \frac{10}{9} \frac{d_E(\gz(t), x^\ast+\rr_+ a^{(i+1)} )}{|\gz(t)-x^\ast|} \\
&  \le  \frac{80}{9}\sqrt{\frac{    \Psi(2|\gz(t)-x^\ast|)} { 2|\gz(t)-x^\ast|}}\\
&\le \frac{80}{9} \wz\Psi( |\gz(t)-x^\ast|).
\end{align*}
Thus for any  $t\in [\sz_{i },\sz_{i+1}]$, it holds that
\begin{align*}| \frac{\gz(t)-x^\ast}{|\gz(t)-x^\ast|}-a |&\le  | \frac{\gz(t)-x^\ast}{|\gz(t)-x^\ast|}-a^{(i+1)}|+ |  a^{(i+1)}-a|\\
&\le  \frac{80}{9}\wz\Psi( |\gz(t)-x^\ast|)+ 4\wz\Psi(2^{i+1})\\
&\le 16 \wz\Psi( |\gz(t)-x^\ast|).
 \end{align*}

{\it Step 3.}
 For  $t\ge \sz_{n_0}$, we have
 $$16 \wz\Psi( |\gz(t)-x^\ast|)\le 16 \wz\Psi(2^{n_0})\le 2^{-6}a^\flat,$$
and hence  $$| \frac{\gz(t)-x^\ast}{|\gz(t)-x^\ast|}-a |\le 2^{-6}a^\flat.$$
 This shows $\frac{\gz(t)-x^\ast}{|\gz(t)-x^\ast|} \in \Omega$ by Lemma \ref{angle}. Write $$ \gz(t)=x^\ast +Tb=x+sa+Tb\ \mbox{with} \ b= \frac{\gz(t)-x^\ast}{|\gz(t)-x^\ast|} \mbox{  and } \ T=|\gz(t)-x^\ast|.$$  By \eqref{cofree}, we see that $(\gz(t))^\flat \ge (x^\ast)^\flat>0$ and $\gz(t) \in \Omega$, that is, \eqref{flatgz} holds.

Recalling \eqref{rel}, i.e. $\sqrt{ t^{-1}\Psi(t)}\le \wz \Psi(t)$, we conclude \eqref{sizeest} from
 \eqref{eqk2}.
\end{proof}
%
%
%
%
%
%
%
%

\section{Large time  behavior of $m_\lz$-geodesics}

Given  any $x\in \wz \Omega$, any
 $a\in \bs^{dN-1}$ with $a^\flat>0$, and any
$\lz>0$, let   $x^\ast$, $\Psi$, $\wz\Psi$  and $n_0$ be as in Section 5.

 \begin{lem}  \label{unifg}
Let $\gz\in \mathcal{AC}(x,x^\ast+2^na;[0,\sz];\wz\Omega)$ be an  $m_\lz$-geodesic with canonical parameter, where
  $n>n_0$. Then
for any $t\ge \sz_{n_ 0}$,  we have
\begin{align}\label{e7-1}
\int_0^{t}|\dot\gz(s)|^2\,ds \le  \sqrt{2\lz}    |\gz(t)-x^\ast|(1 + \wz \Psi^2(|\gz(t)-x^\ast|) )\le
8\lz t,
\end{align}
\begin{align}\label{e7-2} \frac12\le \frac1{1+\wz\Psi^2(|\gz(t)-x^\ast |) }\le
\frac{|\gz(t)-x^\ast |}{ \sqrt{2\lz} t}\le   1+\wz\Psi^2(|\gz(t)-x^\ast |)\le 2 \end{align}
and
\begin{align}\label{e7-3}
\frac{|\gz(t)-x^\ast-\sqrt{2\lz}at|^2  }{ (\sqrt{2\lz}t)^2}\le 2^9\wz\Psi^2(|\gz(t)-x^\ast|)\le 2^9\wz\Psi^2(\frac12\sqrt{2\lz}t).\end{align}

\end{lem}

\begin{proof}
Lemma \ref{key00} implies that
 \begin{equation} \label{almost}
 |\dot\gz(s)|=\sqrt{2(F\circ \gz (s)+ \lz)}, \quad\mbox{ for almost all} \ s\ge 0,
 \end{equation}
 and   hence
 \begin{equation} \label{almost1}  |\dot \gz(s)|^2=\frac 12|\dot \gz(s)|^2+F\circ\gz(s)+\lz, \quad\mbox{ for almost all} \ s\ge 0. \end{equation}

We first show \eqref{e7-2}.  Let  $t\ge \sz_{n_ 0}$. Then
\begin{equation}\label{wzpsit}\mbox{$|\gz(t)-x^\ast|\ge 2^{n_0}$ and hence $\wz \Psi(|\gz(t)-x^\ast|) \le \wz\Psi(2^{n_0})\le 2^{-10}a^\flat \le 2^{-9}$.  }\end{equation}
By \eqref{almost1} and \eqref{sizeest}, one has
\begin{equation}\label{xxxt}\int_{0 }^t|\dot \gz(s)|^2\,ds=A_\lz(\gz|_{[0,t]})=m_{\lz}(x,\gz(t))\le \sqrt{2\lz}|\gz(t)-x^\ast|(1+\wz\Psi^2( |\gz(t)-x^\ast|)).
\end{equation}
Since \eqref{almost} gives $|\dot\gz(t)|\ge \sqrt{2\lz}$,   by \eqref{xxxt} one has
\begin{equation*}
|\gz(t)-x^\ast|[1+\wz\Psi^2( |\gz(t)-x^\ast|)] \ge \sqrt{2\lz} t. \end{equation*}
Note that $\wz \Psi(|\gz(t)-x^\ast|)   \le1$. By \eqref{wzpsit},  one has
$$
\frac{|\gz(t)-x^\ast |} { \sqrt{2\lz} t}\ge [1+\wz\Psi^2(|\gz(t)-x^\ast |]^{-1}\ge \frac12.$$
On the other hand, since
\begin{align*}\frac1t|\gz(t)-x|^2&\le 
 \frac1t\left(\int_{ 0}^t |\dot \gz(s)|\,ds \right )^2
 \le  \int_{0 }^t|\dot \gz(s)|^2\,ds,
\end{align*}
by \eqref{xxxt}, one has
\begin{align*}\frac {|\gz(t)-x|^2}{\sqrt{2\lz} t|\gz(t)-x^\ast| }&\le  1+\wz\Psi^2( |\gz(t)-x^\ast|).
\end{align*}
By \eqref{distanceest} and \eqref{wzpsit},
\[|\frac{\gz(t)-x^\ast}{|\gz(t)-x^\ast|}-a | \le 2^4 \wz\Psi( |\gz(t)-x^\ast|) \le 2^{-5}.\]
It follows that $|\gz(t)-x|\ge |\gz(t)-x^\ast|$ for all $t \ge \sz_{n_ 0}$ (See Figure \ref{f6} for illustration). Since $\wz \Psi(|\gz(t)-x^\ast|)   \le1$ by \eqref{wzpsit},
\begin{align*}\frac {|\gz(t)-x^\ast| }{\sqrt{2\lz} t  }&\le \frac {|\gz(t)-x|^2}{\sqrt{2\lz} t|\gz(t)-x^\ast| } \le  1+\wz\Psi^2( |\gz(t)-x^\ast|)\le 2.
\end{align*}
We therefore obtain \eqref{e7-2}.

By \eqref{xxxt}, \eqref{e7-2} and $\wz \Psi(|\gz(t)-x^\ast|) \le1$,
one has
$$\int_{0 }^t|\dot \gz(s)|^2\,ds \le \sqrt{2\lz}|\gz(t)-x^\ast|(1+\wz\Psi^2( |\gz(t)-x^\ast|))\le 8\lz t,$$
which gives \eqref{e7-1}.

Finally, we prove \eqref{e7-3}.
For $t\ge\sz_{n_0}$,  
by \eqref{angle=dis} one has
$$
  \langle  \frac{\gz(t)-x^\ast}{ |\gz (t)-x^\ast |} ,a\rangle=\cos\angle (\frac{\gz(t)-x^\ast}{ |\gz (t)-x^\ast |},a )=
1- 2\sin^2\frac12\angle (\frac{\gz(t)-x^\ast}{ |\gz (t)-x^\ast |},a )= 1-\frac12 | \frac{\gz(t)-x^\ast}{|\gz(t)-x^\ast|}-a |^2 .$$
By   \eqref{distanceest} and $\Psi(|\gz(t)-x^\ast|) \le 2^{-9}$ as given by \eqref{wzpsit}, one gets
$$
  \langle  \frac{\gz(t)-x^\ast}{ |\gz (t)-x^\ast |} ,a\rangle\ge 1-2^7  \wz\Psi^2(|\gz(t)-x^\ast|)>0  .$$
%
From this and \eqref{e7-2}, we deduce that
\begin{align*}  &\frac{|\gz(t)-x^\ast-\sqrt{2\lz}at|^2   }{ (\sqrt{2\lz}t)^2}\\
&\quad=  \frac{|\gz(t)-x^\ast |^2  }{(\sqrt{2\lz}t)^2}+1-2 \frac { |\gz(t)- x^\ast|} { \sqrt{2\lz}t}  \frac{ \langle  \gz(t)- x^\ast ,a\rangle }{ |\gz(t)- x^\ast|}\\
&\quad\le  [1+\wz\Psi^2(|\gz(t)-x^\ast |) ]^{2}+1- 2  [1+\wz\Psi^2(|\gz(t)-x^\ast |) ]^{-1} [1-2^7  \wz\Psi^2(|\gz(t)-x^\ast|) ] \\
&\quad\le  [1+\wz\Psi^2(|\gz(t)-x^\ast |) ]^{2}+1- 2  [1-\wz\Psi^2(|\gz(t)-x^\ast |) ]  [1-2^7  \wz\Psi^2(|\gz(t)-x^\ast|) ] \\
&\quad\le 2^{9}\wz\Psi^2(|\gz(t)-x^\ast |).  \end{align*}
Since $\Psi(|\gz(t)-x^\ast|) \le 2^{-9}$ and $
 |\gz(t)-x^\ast |\ge \frac12 { \sqrt{2\lz} t}  $,   one has
\begin{align*}  \frac{|\gz(t)-x^\ast-\sqrt{2\lz}at|^2   }{ (\sqrt{2\lz}t)^2} \le 2^{9}\wz\Psi^2(|\gz(t)-x^\ast |)
\le 2^{9}\wz\Psi^2(\frac12\sqrt{2\lz}t).
 \end{align*}
Thus \eqref{e7-3} holds.
The proof is complete.
\end{proof}

\section{Proof of Theorem \ref{mthm}}

Below  we prove Theorem \ref{mthm}
  and hence  Corollary 1.5;
in particular, when  $\mbox{$F=U$ with $f(s)=s^{-1}$}$, we reprove  Theorem 1.3.

\begin{proof}[Proof of Theorem \ref{mthm}]
Given  any $x\in \wz \Omega$, any
 $a\in \bs^{dN-1}$ with $a^\flat>0$, and any
$\lz>0$, let   $x^\ast$, $\Psi$, $\wz\Psi$  and $n_0$ be as in Section 5.  The proofs consists of 5 steps.

\noindent {\it Step 1.}
{\color{black} According to Lemma \ref{ext}, for any $n\in\nn$, let $\gz^{(n)} \in \mathcal {AC}(x,x^\ast+2^{n }a;[0,\sz^{(n)}];\wz\Omega)$ be an $m_\lz$-geodesic with canonical parameter.} For $n\ge j\ge n_0$, set $$\sz^{(n)}_{j}:=\max\{t\in[0,\sz^{(n)}]: |\gz^{(n)} (t)-x^\ast|=2^{j}\}.$$
By \eqref{e7-2}, we have
$$\frac12\le
\frac{|\gz^{(n)} (\sz^{(n)}_j )-x^\ast |}
{ \sqrt{2\lz} \sz^{(n)}_j }= \frac{2^{j}}{\sqrt{2\lz}\sz^{(n)}_j }
\le
 2.$$
 Thus
$$
\frac1{\sqrt{2\lz}} 2^{j-1}\le \sz^{(n)}_j \le \frac1{\sqrt{2\lz}} 2^{j+1}.
$$
Moreover,   $ t\ge \frac1{\sqrt{2\lz}}2^{n_0+1}$ implies that $t\ge \sz^{(n)}_{n_0} $ for all possible {\color{black}$n\ge n_0$}.
Thus, {\color{black}by \eqref{flatgz} from Lemma \ref{unif},}
we have $$\mbox{$(\gz^{(n)}(t))^\flat\ge (x^\ast)^\flat>0$ whenever $ \frac1{\sqrt{2\lz}}2^{n_0+1}\le t\le \sz^{(n)}$ for all possible {\color{black}$n\ge n_0$}. }
$$

\noindent {\it Step 2.} We show that for each $j\ge n_0$,   the family $\{\gz^{(n)} |_{[0,\frac1{\sqrt{2\lz}}2^j]}\}_{n\ge j+1}$ 
is
 uniformly bounded and {\color{black}equi-continuous}. Indeed,
by \eqref{eqk2}, {\color{black}\eqref{distanceest}} and $\Psi(2^j)\le 1$, one has
$$l(\gz^{(n)} |_{[0,\sz^{(n)}_j ]})\le \frac1{\sqrt {2\lz}}m_\lz(x,\gz^{(n)} (\sz^{(n)}_j ))\le 2^{j}+\Psi(2^{j})\le 2^{j+1}, \quad\forall j\ge  n_0 .$$
 By \eqref{e7-1},
$$
 \int_0^{\sz^{(n)}_j }|\dot \gz^{(n)} (s)|^2\,ds \le  8\lz \sz^{(n)}_j \le 4\sqrt{2\lz}2^{j+1}.$$
Thus one has
$$l(\gz^{(n)} |_{[0,\frac1{\sqrt{2\lz}}2^j]})\le l(\gz^{(n)} |_{[0,\sz^{(n)}_{j+1}]})\le 2^{j+2} \quad\mbox{ and hence}\quad \gz^{(n)} |_{[0,\frac1{\sqrt{2\lz}}2^j]}\subset B(x,2^{j+2}).$$
While  for any $0\le s<t\le \frac1{\sqrt{2\lz}} 2^{j } $, we have
$$|\gz^{(n)} (t)-\gz^{(n)} (s)|\le \int_s^t |\dot\gz^{(n)} (\dz)|\,d\dz \le |t-s|^{1/2}\left({\color{black}\int_0^ {\sz^{(n)}_{j+1}}} |\dot\gz^{(n)} (\dz)|^2\,d\dz \right)^{1/2}
\le 4\cdot 2^{j/2}(2\lz)^{1/4} |t-s|^{1/2}.$$

\medskip
\noindent {\it Step 3.} For $j=n_0$, by Arzela-Ascoli theorem,
 there  is  {\color{black}some infinite subset $\mathbb N_j\subset\mathbb N_{j-1}$} such  that {\color{black}the subsequence}
$$\{\gz^{(n)} |_{[0,\frac1{\sqrt{2\lz}}2^{n_0 } ]}\}_{n\in\nn_{n_0}}  \mbox{ converges uniformly to some
curve } \ \eta_0\in C^0([0,\frac1{\sqrt{2\lz}}2^{n_0 }];\rr^{dN})$$
For any $j>n_0$, we can find {\color{black}some infinite subset $\mathbb N_j\subset\mathbb N_{j-1}$}  such that {\color{black}the subsequence}
$$\{\gz^{(n)}  |_{[0,\frac1{\sqrt{2\lz}} 2^j]  }\}_{n \in \mathbb N_j} \ \mbox{ converges uniformly to some curve}\ \eta_j\in C^0([0,\frac1{\sqrt{2\lz}}2^{j }];\rr^{dN}).$$
Since  $\mathbb N_j\subset\mathbb N_{j-1}$ we know that  $$ \eta_j|_{[0, \frac1{\sqrt{2\lz}} 2^{j-1}]}=\eta_{j-1}.$$   This allows us to define a ray
$\gz:[0,\fz)\to\Omega$ via  $$\gz|_{[0,\frac1{\sqrt{2\lz}} 2^j]}:=\eta_j\quad\forall j\ge n_0.$$

We also observe that for any $t\ge \frac1{\sqrt{2\lz}}2^{n_0+1}$, if $j$ is large enough such that $t\le \frac1{\sqrt{2\lz}}2^{j-1}$, then
$$ (\gz (t))^\flat=(\eta_j (t))^\flat= (\lim_{\nn_j\ni n\to\fz}   \gz^{(n)}(t))^\flat =\lim_{\nn_j\ni n\to\fz}  ( \gz^{(n)}(t))^\flat  \ge (x^\ast)^\flat>0.$$
Thus
\begin{equation}\label{gzomega} \gz(t)\in\Omega \quad \mbox{for all}\quad    t\ge \frac1{\sqrt{2\lz}}2^{n_0+1}.
\end{equation}

\medskip

\noindent {\it Step 4.} We show that $\gz$ is an $m_\lz$-geodesic ray with canonical parameter.
To see this, by Lemma \ref{ext}, it suffices to prove that for all
 $j\ge n_0$, $\eta_j$ is an $m_\lz$-geodesic  with canonical parameter, that is,
$\eta_j\in \mathcal {AC}  \left(x,\gz(\frac1{\sqrt{2\lz}}2^j);[0, \frac1{\sqrt{2\lz}}2^j];\wz \Omega \right)$ and  $A_\lz(\eta_j)= m_\lz \left(x,\eta_j(\frac1{\sqrt {2\lz}}2^j) \right)$.

Indeed, since $\frac1{\sqrt{2\lz}}2^j\le \sz^{(n)}_{j+1}$, {\color{black}we can apply Tonelli's Theorem for convex Lagrangians to get
$$ \frac12 \int_0^{\frac1{\sqrt{2\lz}}2^j} |\dot \eta_j(s)|^2 \, ds  \le \frac12 \liminf_{n \to \fz} \int_0^{\frac1{\sqrt{2\lz}}2^j} |\dot \gz^{(n)} (s)|^2 \, ds $$
and Fatou's Lemma to obtain that
$$   \int_0^{\frac1{\sqrt{2\lz}}2^j} F(\eta_j(s)) \, ds  \le  \liminf_{n \to \fz} \int_0^{\frac1{\sqrt{2\lz}}2^j} F(\gz^{(n)}(s) ) \, ds. $$
Therefore, we conclude
\begin{equation}\label{Aeta}
A_\lz(\eta_j)\le \liminf_{n \in \nn_j,  n\to\fz} A_\lz \left(\gz^{(n)} |_{[0,\frac1{\sqrt{2\lz}}2^j]} \right) \le
  \liminf_{n \in \nn_j, n\to\fz} m_\lz \left(x, \gz^{(n)} ( \sz^{(n)}_{j+1}) \right)\le \sqrt{2\lz} 2^{j+2}.
\end{equation} }
Thus $\eta_j\in \mathcal {AC} (x,\gz(\frac1{\sqrt{2\lz}}2^j);[0, \frac1{\sqrt{2\lz}}2^j];\wz \Omega )$.
Obviously, $m_\lz(x,\eta_j(\frac1{\sqrt {2\lz}}2^j))\le A_\lz(\eta_j)$.

Below we show that $\displaystyle A_\lz(\eta_j)\le m_\lz\left(x,\eta_j (\frac1{\sqrt {2\lz}}2^j)\right)$.
Since $\gz^{(n)}$ is an $m_\lz$-geodesic with canonical parameter, one has
$$ A_\lz \left(\gz^{(n)} |_{[0,\frac1{\sqrt{2\lz}}2^j]} \right)= m_\lz \left(x, \gz^{(n)}(\frac1{\sqrt {2\lz}}2^j) \right)\le
m_\lz \left(x, \gz (\frac1{\sqrt {2\lz}}2^j) \right) + m_\lz \left(\gz (\frac1{\sqrt {2\lz}}2^j), \gz^{(n)} (\frac1{\sqrt {2\lz}}2^j) \right).$$
By Lemma \ref{klem}, write $z_{n,j}=\gz (\frac1{\sqrt {2\lz}}2^j)- \gz^{(n)} (\frac1{\sqrt {2\lz}}2^j)$ one has
$$m_\lz \left(\gz (\frac1{\sqrt {2\lz}}2^j), \gz^{(n)} (\frac1{\sqrt {2\lz}}2^j) \right)\le
\sqrt{2\lz}|z_{n,j}| + {\color{black}\frac{1}{\sqrt{2\lz}}}\int_{0}^{|z_{n,j}|}F \left(\gz (\frac1{\sqrt {2\lz}}2^j)+s\frac{ z_{n,j} }{|z_{n,j}|} \right)\,ds. $$
Since $F$ is bounded in any compact suhset of $\Omega$,    noting $|z_{n,j}|\to 0$ as $\nn_j\ni n\to\fz$ we know that
$$\lim_{n \in \nn_j, n\to\fz} m_\lz \left(\gz (\frac1{\sqrt {2\lz}}2^j), \gz^{(n)} (\frac1{\sqrt {2\lz}}2^j) \right)=0.$$
From this and \eqref{Aeta}, we conclude that
$$ A_\lz(\eta_j)\le \limsup_{n \in \nn_j, n\to\fz}A_\lz \left(\gz^{(n)} |_{[0,\frac1{\sqrt{2\lz}}2^j]} \right) \le
m_\lz \left(x, \gz (\frac1{\sqrt {2\lz}}2^j) \right) $$
as desired.

\medskip
\noindent {\it Step 5.}
For any $t>\frac1{\sqrt {2\lz}}2^{n_0}$,
letting $j>n_0$  such that  $  t<\frac1{\sqrt{2\lz}}2^{j}$,  by \eqref{e7-3}  one has
\begin{align}\label{lh}
\frac{|\gz(t)-x^\ast-\sqrt{2\lz}at|    }{  \sqrt{2\lz}t }&=\frac{|\eta_j(t)-x^\ast-\sqrt{2\lz}at|    }{  \sqrt{2\lz}t }\nonumber\\
&=\lim_{n \in \nn_j, n\to\fz}
\frac{|\gz^{(n)} (t)-x^\ast-\sqrt{2\lz}at|   }{  \sqrt{2\lz}t }\nonumber \le 2^5 \wz \Psi(\frac12\sqrt{2\lz}t) .  \end{align}
Hence
$$\lim_{t\to\fz}
\frac{|\gz(t)- \sqrt{2\lz}at|^2   }{ (\sqrt{2\lz}t)^2} = \lim_{t\to\fz}
\frac{|\gz(t)-x^\ast-\sqrt{2\lz}at|^2   }{ (\sqrt{2\lz}t)^2}=0,$$
That is,  Theorem \ref{mthm} (i) holds.

Theorem \ref{mthm} (iv) follows from Theorem \ref{mthm} (i) and Lemma A.4.
Moreover, Theorem \ref{mthm} (ii) follows from \eqref{gzomega} with $t_0=\frac1{\sqrt{2\lz}}2^{n_0+1}$.
Theorem \ref{mthm} (iii) follows from  Theorem \ref{mthm} (i), Theorem \ref{mthm} (ii)  and Lemma A.4.

We end the proof of Theorem \ref{mthm}.
\end{proof}

\renewcommand{\thesection}{Appendix A}
 \renewcommand{\thesubsection}{ A }
\newtheorem{lemapp}{Lemma \hspace{-0.15cm}}
\newtheorem{corapp}[lemapp] {Corollary \hspace{-0.15cm}}
\newtheorem{remapp}[lemapp]  {Remark  \hspace{-0.15cm}}
\newtheorem{defnapp}[lemapp]  {Definition  \hspace{-0.15cm}}
\newtheorem{egapp}[lemapp]  {Example  \hspace{-0.15cm}}
\renewcommand{\theequation}{A.\arabic{equation}}

\renewcommand{\thelemapp}{A.\arabic{lemapp}}

\section{Properties of Hamiltonians and Ma\~{n}\'{e}'s potentials}

In the appendix we always assume that $m_i=1$ for all $1\le i\le N$. {\color{black} For general masses, all of the following conclusions still hold, up to some obvious modifications.} We omit the details.

Let $F$ be as in \eqref{F}. Apriori, $F$ is only defined in the set $ \Omega$ since
 $F_{ij}$ in only defined in $\rr^{d}\times\rr^d\setminus\Delta$ for all $1\le i<j\le \fz$.
Now  we extend the definition of  $F$  to the whole $ \rr^{dN}$ by defining the value of
$F_{ij}$ for all $1\le i<j\le \fz$ as below:
$$F_{ij}(z,z)=\liminf_{(x_i,x_j)\to(z,z)}F_{ij}(x_i,x_j), \quad \forall (z,z)\in\Delta.$$
It may happen that $F_{ij}(z,z)=\fz$ for some $z\in\rr^d$, and hence $F(x)=\fz$ for some $x\in\Sigma$.
But one may directly check that  that  $F_{ij}$ is   lower semicontinuous  in $\rr^d\times\rr^d$, that is,
the set
$\{(x_i,x_j)\in\rr^d\times\rr^d: F_{ij}(x_i,x_j)>\dz\}$ is open for all $\dz>0$.
Thus $F$ is also lower semicontinuous  in whole $\rr^{dN}$.

For any $\lz>0$, Ma\~{n}\'{e}'s potential $m_\lz$ {\color{black} is defined in} \eqref{mane2}.
Observe that    \begin{equation}\label{ew.1}
 \mbox{$m_\lz(x,y)\ge \sqrt{2\lz}|x-y|$ for all $x,y\in\rr^{dN}$.}
\end{equation}Indeed,
for any $\gz\in \mathcal {AC}(x,y;[0,\sz],\rr^{dN})$,
by the Cauchy-Schwartz inequality,    its Euclidean length
$$ l(\gz) \le (2\sz  )^{1/2}\left(\int_0^\sz  \frac12|\dot\gz(s)|^2 \,ds\right)^{1/2}\le  (2\sz )^{1/2}\left(A_\lz-\lz \sz \right)^{1/2},$$
and hence by $|x-y|\le l(\gz)$, one has
$$A_\lz(\gz)\ge  \frac1{2\sigma }  l(\gz) ^2  +\lz \sigma \ge \frac1{2\sigma  }  |x-y| ^2  +\lz \sigma.$$
Since  $\frac1{2\sigma }  |x-y| ^2  +\lz \sigma$ {\color{black}reaches its minimal value} at $\sigma= |x-y|/\sqrt {2\lz}$, we have $A_\lz(\gz)\ge \sqrt{2\lz}|x-y|$  as desired.

 On the other hand,  we have
$$m_\lz(x,y) <\fz, \, \, \quad\forall x,y\in\Omega.$$
Indeed, since $\Omega$ is path connected, one can always find a smooth curve $\gz_{x,y}\in \mathcal{AC}(x,y;[0,\sz];\Omega)$ for some $\sz>0$.
By the semicontinuity of $F$  and  $F<\fz$ in $\Omega$, we know that
 $F$ is bounded in the compact set $\gz([0,\sz])$ and hence we know that  $A_\lz(\gz)<\fz$.
But  for  $x\in\Omega $ and $y\notin\Omega$, it is not clear whether  $ m_\lz(x,y) $ is finite or not.
This is indeed determined by the behaviour of $F$ around  $\Sigma$.

Set  $$\Omega_\lz:=\{y\in\rr
^{dN}: m_\lz (x,y)<\fz, \ \forall x\in\Omega\}.$$
Obviously, $\Omega \subset \Omega_\lz\subset\rr^{dN}$.
We observe that  $\Omega_\lz=\Omega_\mu$ for all $0<\lz<\mu<\fz$. To see this, it suffices to show that
$$\mbox{$m_\lz(x,y)<\fz$ if and only if  $m_\mu(x,y)<\fz$  for any $x\in\Omega,y\in\rr^{dN}$.}$$
Since $A_\lz(\gz)\le A_\mu(\gz)$ for all possible $\gz$, by definition, one always has  $ m_\lz(x,y)\le m_{\mu}(x,y)$.
On the other hand,  if $m_\lz(x,y)<\fz$, we have $A_\lz(\gz)<\fz$ for some $\gz\in \mathcal {AC}(x,y;[0,t];\rr^{dN})$. Thus   $$A_\mu(\gz)\le A_\lz(\gz)+ (\mu-\lz) t<\fz,$$   which implies $m_\mu(x,y)<\fz$ as desired.  Recalling \eqref{wzoz}, we write $\wz\Omega=\Omega_\lz$ for any $\lz>0$.

{\color{black} One may directly check that  $ m_\lz$ is a distance 
 in $\wz\Omega$, and hence   
  $(\wz \Omega,m_\lz)$ is a  metric space.
 Observe that  $(\wz \Omega,m_\lz)$ is always   complete.
Indeed, let $\{y_n\}_{n\in\nn}\subset \wz \Omega $ be a Cauchy sequence with respect to $m_\lz$. 
Recall that \eqref{ew.1} gives $|z-w|\le \frac1{\sqrt{2\lz}}m_\lz(z,w)$ for all $z,w\in\rr^{dN}$.
It follows that $\{y_n\}_{n\in\nn}$ is also a Cauchy sequence with respect to the Euclidean distance, and hence $|y_n-y|\to 0$ as $n\to\fz$ for some $y\in\rr^{dN}$.
It then suffices to show that  $y\in\wz\Omega$ and $m_\lz(y_n,y)\to 0$ as $n\to\fz$.
{\color{black}If there exists $m\in\nn$ such that  $y_n=y$ for all $n\ge m$, then $y_m\in \wz\Omega$ implies $y\in\wz \Omega$ and $m_\lz(y_n,y)=0$ for all $n\ge m$ as desired.}
Now we assume that there are infinitely many $k$ such that $y_k\ne
y$.  Thanks to this, 
it is  standard to  find a subsequence $\{y_{n_k}\}_{k\in\nn}$ of $\{y_n\}_{n\in\nn}$
such that $0<m_\lz(y_{n_k},y_{n_{k+1}})\le 2^{-k}$.  For each $k\in\nn$,
there is  $ \gz_k \subset \mathcal {AC}(y_{n_k},y_{n_{k+1}};[0,\sz_k];\rr^{dN})$  such that
$A_\lz(\gz_k)\le 2m_\lz(y_{n_k},y_{n_{k+1}}) \le 2^{-k+1}$.
Fix an arbitrary $x \in \Omega$.  There is
  $ \gz_0  \subset \mathcal {AC}(x,y_{n_1};[0,\sz_0];\rr^{dN})$ such that
  $A_\lz(\gz_0)\le 2m_\lz(x,y_1) <\fz$.
Noting that $\sz_k\le \frac1\lz A_\lz(\gz_k)$ for all $k\in\nn\cup\{0\}$, it implies that
\[  \sz:=\sum_{k=0}^\fz\sz_k \le \frac1\lz\sum_{k=0}^\fz A_\lz(\gz_k) \le \frac1\lz\sum_{k=0}^\fz  2^{-k+1}<\fz.\]  Then the concatenation  $\gz$ of these curves $\gz_k$  (that is,
 {\color{black}$\gz=\ast_{k}\gz_k$)} belongs to $\mathcal {AC}(x,y;[0,\sz];\rr^{dN})$ and satisfies
 $A_\lz(\gz)\le \sum_{k=0}^\fz A_\lz(\gz_k)<\fz$. Thus   $m_\lz(x,y)<\fz$, that is, $y\in\wz \Omega$, and moreover $$m_{\lz}(y_{n_k},y)\le \sum_{j=k}^\fz A_\lz(\gz_j)
 \le  \sum_{j=k}^\fz  2^{-j+2}=  2^{-k+3},$$
 which implies $m_{\lz}(y_{n_k},y)\to0$ as $k\to\fz$.
 Since  $\{y_n\}_{n\in\nn}$ is a Cauchy sequence with respect to $m_\lz$,
 by  a standard argument, one further gets $m_\lz(y_n,y)\to 0$ as $n\to\fz$ as desired.
 }


{\color{black}    Recall that $\Omega\subset \wz\Omega$, and $\Omega$ is an open subset of $\wz\Omega$ with respect to $m_\lz$.
Thus   $(\Omega, m_\lz)$   is  also a metric space as a subspace of    $(\wz \Omega,m_\lz)$.
Since there is no any other assumption made on the behaviour  of $F$ around $\Sigma$ (except the lower semicontinuity), $\Omega$  is not necessarily closed
or complete in general. It is interesting to determine the relation between $(\Omega,m_\lz)$,  or its completion/closure,  and $(\wz \Omega, m_\lz)$.
 In the following remark, we have some discussion about this issue. 

\begin{remapp} \label{eg1} \rm
(i) It is easy to see that  if
\begin{align*}
  \mbox{$  F_{ij}(x_i,x_j) \ge  C |x_i-x_j|^{-2}$  when $|x_i-x_j|\le \frac12$ for all $1\le i<j\le N$,}
\end{align*}
then $\wz\Omega=\Omega$, and hence,  $(\Omega,m_\lz)$ is complete.  One may also check that if
\begin{align*}
   \mbox{ $| F_{ij}(x_i,x_j)|\le C|x_i-x_j|^{-\alpha}$  when $|x_i-x_j|\le \frac12$ for all $1\le i<j\le N$ and $\alpha \in (0, 2)$,}
\end{align*}
then $\Omega \subsetneqq \wz\Omega=\rn$.
But in general, it is difficult to determine the set $\wz\Omega$. For example, see \cite{ch02, cm00, m02, ft04, ch08, bft08, btv13, mo18} and references therein.

(ii) If  $F=+\fz$ in $\Sigma= \rr^{dN}\setminus \Omega$, 
then $\Omega$ is dense in $(\wz\Omega,m_\lz)$, equivalently,
the completion of $( \Omega,m_\lz)$ is  $(\wz \Omega, m_\lz)$.
 Indeed, if $\Omega \subsetneqq \wz \Omega$, given any  $y \in \wz \Omega \setminus \Omega$ and
 $x \in \Omega$, there is $ \gz   \in \mathcal {AC}(x,y ;[0,\sz];\rr^{dN})$ with $A_\lz(\gz)<\fz$. Since $F=\fz$ in $\Sigma$, we know that
 $\gz(t)\subset\Omega$  for almost all $t\in[0,\sz]$. Therefore there exists a sequence $\{t_n\}\subset [0,\sz]$ so that $t_n\to \sz$ with $\gz(t_n)\in\Omega$, and  $A_\lz(\gz|_{[t_n,\sz]})\to 0$, which implies  $m_\lz(\gz(t_n), y) \to 0$ as $n \to \fz$, that is, $y$ is the limit of the Cauchy sequence $\{\gz(t_n)\} \subset \Omega$ with respect to $m_\lz$.
 This shows the closure $\overline \Omega^{m_\lz}$ of $\Omega$ with respect to $m_\lz$ contains $\wz \Omega$. On the other hand, since $\Omega
 \subset\wz\Omega$ and $\wz \Omega$  is complete with respect to $m_\lz$, it follows that $\overline \Omega^{m_\lz} \subset \wz \Omega$. We conclude that
  $\overline \Omega^{m_\lz}=\wz \Omega$.

(iii)    Without  the assumption $F=+\fz$ in $\Sigma$, in general
 one can not expect that   $\Omega$ is dense in $(\wz\Omega,m_\lz)$, equivalently, the completion of $( \Omega,m_\lz)$ is  $(\wz \Omega, m_\lz)$.
Indeed, we construct a   potential $F:\rr^{dN}\to (0,\fz)$
so that $\Omega$ is not dense in $(\wz\Omega,m_\lz)$; see the following  3 steps.

\medskip
{\bf Step 1.}  Constructions of   $h_\pm:[0,\fz)\to(0,\fz)$ by modifying $r^{-2}$.\\

Let   $h_\pm:[0,\fz)\to(0,\fz)$ be  two   functions   defined by
\begin{equation*}
h_+(r)=\left\{
\begin{split}
  1\ \quad\quad & \mbox{when}\ 0\le r\le 1;\\
 \frac1{r^2}\quad\quad& \mbox{when}\ 1<r<\fz
\end{split}\right.
\end{equation*}
and
\begin{equation*}
h_-(r)=\left\{
\begin{split}
&1 \quad& \mbox{when}  \ r\in \Lambda:=\{0\}\cup(\cup_{k\ge3}  [2^{-k^2-1}, 2^{-k^2+1}]) ; \\
&\frac1{r^2}\quad& \mbox{ when }\ r\in(0,\fz)\setminus \Lambda.
\end{split}\right.
\end{equation*}
 Obviously, $h_+\in C^0([0,\fz))$, and
    $ h_-$ is lower semicontinuous in $[0,\fz) $ with
   $$\displaystyle h_-(r)=\liminf_{t\to r}h_-(t)\quad\forall \ 0\le r<\fz.$$
moreover, $h_+(r)\le h_-(r)\le r^{-2}$ for all $0<r<\fz$, and
 $h_+=h_-$ in $\Lambda \cup [1,+\fz)$.

\medskip
  {\bf Step 2.} 
Construction of a potential  $F:\rr^{dN}\to(0,\fz)$  via  $h_\pm$.\\

 For any nonempty subset $E \subset \rr^{2d}$, recall the standard characteristic function $\chi_{E}: \rr^{2d} \to \{0, 1\}$, which maps the elements of $E$ to $1$, and all other elements to $0$. Define
 \begin{equation*}
F(x):= F_{12}(x_1,x_2) \mbox{ for all  $x=(x_1, \cdots, x_N) \in \rr^{dN}$}
\end{equation*}
with
 $$F_{12}(x_1,x_2):= h_-(|x_1-x_2|)\chi_{ \Gamma \times \Gamma}+
  h_+(|x_1-x_2|)\chi_{\rr^{d}\times \rr^d \setminus [\Gamma \times \Gamma]},$$
where
$$ \mbox{$x_i=(x_i^1, \cdots, x_i^d)  \in \rr^{d }$ and } \Gamma:= \{x_i \in \rr^{d } :  x_i^j < 0 , \ \forall 1 \le j\le d \}.$$

Obviously,   $F_{12}$ is locally bounded   in $\rr^{d}\times\rr^d\setminus\{(x_1,x_2)\in\rr^{d}\times\rr^d : x_1=x_2, \, x_1\in\overline{\Gamma}\},$
and $F_{12}(x_1,x_2)= 1$ whenever $  |x_1-x_2|\in   \Lambda$. Moreover, $F_{12}$ is lower semicontinuous in $\rr^d\times\rr^d$ with
$$F_{12}(x_1,x_2)=\liminf_{(z_1,z_2)\to(x_1,x_2)}F_{12}(z_1,z_2) \mbox{ for all
$ (x_1,x_2)\in\rr^{2d}$}.$$
In particular,  $F$ satisfies the condition \eqref{F}.

 Note that
$$\limsup_{(z_1,z_2)\to (x_1,x_2)} F_{12}(z_1,z_2)=+\fz\quad
\mbox{whenever $x_1=x_2$ and $x_1\in\Gamma $.}$$
Below we write
$$\Sigma^-_{12}:=\{x=(x_1,\cdots,x_N)\in\rr^{dN}: x_1=x_2\in\Gamma\} \subset \Gamma\times\Gamma\times\rr^{d(N-2)} $$
and
$$\Sigma_{12}:=\{x=(x_1,\cdots,x_N)\in\rr^{dN}: x_1=x_2\}. $$
{\color{black} Notice that the local boundedness of $F_{12}$ in $\rr^{d}\times\rr^d\setminus\{(x_1,x_2)\in\rr^{d}\times\rr^d : x_1=x_2, \, x_1\in\overline{\Gamma}\}$ implies the local boundedness of $F$ in $\rr^{dN}\setminus \overline{\Sigma_{12}^-}$. In particular, $F$ is locally bounded in $\Sigma \setminus \Sigma_{12}$. }

\medskip
{\bf Step 3.}  We show that
$\wz\Omega =\rr^{dN}$ and $ \Sigma_{12}^{-}\subset  \wz\Omega\setminus \overline \Omega^{m_\lz}$.\\

We  first show that  $\wz\Omega =\rr^{dN}$. It suffices to show that $\Sigma \subset\wz\Omega$.
We consider three cases: $x \in \Sigma \setminus \Sigma_{12}$, $x \in \Sigma_{12} \setminus \overline{\Sigma_{12}^-}$
and $x \in  \overline{\Sigma_{12}^-}$.

\medskip
{\color{black}{\it   Case   $x \in \Sigma \setminus \Sigma_{12}$.} By Step 2, we know that $F$ is locally bounded at $x$. Since $\rr^{dN}\setminus \Sigma_{12}$ is a connected open subset, there exists an Euclidean ball $B(x,r_x) \subset \rr^{dN}\setminus  \Sigma_{12}  $ for some $r_x>0$ such that $F|_{B(x,r_x)} \le M_x$ for some $M_x>0$.}
  Noting that   $\Sigma=\rr^{dN}\setminus\Omega$ is closed and does not have interior points, we can find {\color{black} a curve $\gz_x \in C^1([0,t] , B(x,r_x)) $} with end point $x$ such that $\gz_x \setminus \{x\} \subset \Omega$ and {\color{black}$F|_{\gz_x}  \le M_x$.} Therefore, we  see that
$m_\lz(\gz_x(0),x)\le A_\lz(\gz_x)<\fz$, that is, $x \in \wz \Omega$.

\medskip
{\color{black}{\it   Case   $x \in \Sigma_{12} \setminus \overline{\Sigma_{12}^-}$.} By the definition of $F$, we know $F(x)=1$. Note that $\rr^{dN}\setminus \overline{\Sigma_{12}^-}$ is a connected open subset of $\rr^{dN}$. Similar to the first case, there is an Euclidean ball $B(x,r_x) \subset \rr^{dN}\setminus \overline{\Sigma_{12}^-}$ for some $r_x>0$ such that $F|_{B(x,r_x)} \le 2$, and moreover,
 we can find a curve $\gz_x \in C^1([0,t] , B(x,r_x)) $ with end point $x$ such that $\gz_x \setminus \{x\} \subset \Omega$ and $F|_{\gz_x}  \le 2$. Therefore, we
 see  that
$m_\lz(\gz_x(0),x)\le A_\lz(\gz_x)<\fz$, that is, $x \in \wz \Omega$.}

\medskip
{\it Case  $x \in   \overline{\Sigma_{12}^-}$.} Letting $y = (y_i)_{1}^N \in   \Sigma_{12}\setminus \overline{\Sigma_{12}^-}$ such that $y_1=y_2=(1,\cdots,1) \in \rr^d$ and $y_i=x_i \in \rr^d$ for $i=3,\cdots,N$, taking $\gz$ as the line-segment joining $y$ and $x$, then   $\gz\subset \Sigma_{12}$ and hence
$  F|_{\gz}\equiv 1$.
We have $m_\lz(x,y)<A_\lz(\gz)<\fz$. Since $y \in \Sigma_{12}\setminus \overline{\Sigma_{12}^-} \subset \wz\Omega$, there exists some $x_0 \in \Omega$, such that $m_\lz(x,x_0) \le m_\lz(x,y) + m_\lz(y, x_0)<\fz$. Hence, $x \in \wz\Omega.$

Finally, we show that
$\Sigma_{12}^{-}\subset \wz \Omega\setminus \overline \Omega^{m_\lz}$.
Let  $x\in \Sigma_{12}^{-}$ be an arbitrary point.
It then suffices to show that
\begin{equation}\label{ea.x1}
 \inf_{z\in\Omega} m_\lz(z,x)   >0.
\end{equation}
 To prove this, we need   some properties of $m_\lz$-geodesics joining $z,x$ which will be introduced/proved below.
So we postpone the  proof of  \eqref{ea.x1} to the end of this appendix.
\end{remapp}}

Moreover,   we know that $(\wz\Omega, m_\lz)$ {\color{black} is a geodesic space}, that is,  for any
$x , y \in \wz \Omega $ with $x\ne y$,  there exists a
$\eta\in \mathcal {AC}(x, y; [0,m_\lz(x,y)];\rr^{dN})$ such that
$$
\mbox{$\eta([0,m_\lz(x,y)])\subset\wz\Omega$,
and
$m_\lz(\eta(s),\eta(t))=t-s$  for all $0\le s<t\le m_\lz(x,y)$.}$$
Indeed, thanks to Lemma \ref{ext} below, the desired geodesic  $\gz$   comes from the arc length (with respect to $m_\lz$) parametrisation of  the following minimizer of
$A_\lz$ in the class of $\cup_{\sz>0}\mathcal {AC}(x, y; [0,\sz];\wz\Omega)$.
 \begin{lemapp}  \label{ext}
Given $\lz>0$ and $x , y \in \wz \Omega $ with $x\ne y$,
there is a curve $\gz  \in \mathcal {AC}(x, y; [0,\sz];\wz\Omega)$ such that $A_{\lz}(\gz ) = m_\lz(x, y)$.
Moreover,  for any $\gz  \in \mathcal {AC}(x, y; [0,\sz];\wz\Omega)$ satisfying $A_{\lz}(\gz ) = m_\lz(x, y)$, we have
 $$m_\lz(\gz(s ),\gz(t))=A_\lz(\gz|_{[s ,t ]}), \quad\forall 0\le s <t\le \sz.$$

\end{lemapp}
The proof of Lemma A.2  is standard, for  readers' convenience we give it  later.

Considering the geodesic nature,  as in the introduction we call a minimizer $\gz:[0,\sz]\to\wz\Omega$ of $A_\lz$ with given endpoints $x,y$  as an {\it $m_\lz$-geodesic with the canonical parameter} joining $x,y$.
 We also call a  ray  $\gz:[0,\fz)\to\wz\Omega$ as an {\it $m_\lz$-geodesic ray  with the canonical parameter} if
$\gz|_{[0,\sz]}$ is an $m_\lz$-geodesic with the canonical parameter  for any $\sz>0$.

Next we prove the following result.
\begin{lemapp}\label{key00}
Given any $x,y\in\wz\Omega$ and $\lz>0$, let
  $\gz\in \mathcal {AC}(x,y;[0,\sz], \rr^{dN})$ be any
 $m_\lz$-geodesic    with canonical parameter. Then $\gz$   has energy constant $\lz $, that is,
 $$ |\dot \gz(s)|=\sqrt{2(F(\gz(s))+\lz)} \mbox{ almost all $s\in[0,\sz]$.  }$$
Moreover,  $$
m_\lz(x,y)= \wz m_\lz(x,y)=\wz A_\lz(\gz), 
$$
where  \begin{equation*} 
\wz m_\lz(x,y):=\inf\{ \wz A_\lz (\eta):\eta \in \cup_{\tau>0}\mathcal {AC}(x,y;[0,\tau], \rr^{dN})\}\quad\forall x,y\in\rr^{dN},
\end{equation*}
and   \begin{equation*} 
\wz A_\lz(\eta):=\int_0^{\tau} |\dot \eta(s)|\sqrt{2 F\circ \eta+2 \lz }\,ds.\end{equation*}
{\color{black}Here in the above inequality, we use the convention that $|\dot \eta(s)|\sqrt{2 F\circ \eta+2 \lz} =0$ when $|\dot \eta(s)|=0$ and $F\circ \eta(s) =\fz$.}
\end{lemapp}


Consequently, we have the following.
\begin{lemapp}\label{A6}
Under additional assumption $F\in C^{2}(\Omega)$,  for any $\lz>0$,
if
  $\gz\in \mathcal {AC}(x,y;[0,\sz], \rr^{dN})$ is   an
 $m_\lz$-geodesic    with canonical parameter joining $x,y$  and also is   interiorly collision-free (that is,
 $\gz|_{(0,\sz)}\subset\Omega$),  then  $\gz$ is a
 solution to $\ddot x=\nabla F$ in $(0,\sz)$  starting from $x$ and ending at $y$.

Consequently,  if
$\gz:[0,\fz)\to\rr^{dN}$ is an $m_\lz$-geodesic ray   with canonical parameter starting from $x $  and also is interiorly collision-free (that is,
 $\gz|_{(0,\fz)}\subset\Omega$), then    $\gz$ is a
 solution to $\ddot x=\nabla F$ in $(0,\fz)$  starting from $x$.

%

\end{lemapp}

Below we prove Lemma A.2-Lemma A.4.
\begin{proof}[Proof of Lemma \ref{ext}]
Let $x, y \in \wz \Omega $ be two given configurations, with
$x \ne y$. Since  $0<m_\lz(x,y)<\fz$, there exist  a sequence of curves $\{ \dz_n \in\mathcal {AC}(x,y;[0,\sigma_n];\rr^{dN})\}_{n\in\nn}$ such that
$$   \lim_{n \to \fz} A_\lz(\dz_n) = m_\lz(x,y).$$
Thus for $n$ large enough we have $A_\lz(\dz_n) \le m_\lz(x,y)+1$.
Without loss of generality, we may assume that for all $n$, this holds.
This also implies that $\dz_n \in\mathcal {AC}(x,y;[0,\sigma_n];\wz \Omega)$; otherwise $\dz_n(t)\notin \wz\Omega$ for some $0<t<\sigma_n$ and hence
$$m_\lz(x,\dz_n(t))<A_\lz(\dz_n|_{[0,t]})<A_\lz(\dz_n )<\fz,$$
which is a contradiction with the definition of $\wz\Omega$.

 Next, for every $n$, since $$\frac1{2\sigma_n}  |x-y| ^2  +\lz \sigma_n\le  A_\lz(\dz_n) \le 1+m_\lz(x,y), $$   we know
$$0< \liminf_{n\to\fz} \sz_n \le \limsup_{n \to \fz}\sz_n < +\fz. $$
and therefore up to considering subsequence, we
 may assume that $\sz_n \to \sz_0$ as $n \to \fz$.

Now,  for each $n > 0$, we  parameterize $\dz_n$ linearly as below.
Define  $\gz^{(n)} (t) = \dz_n(\sz_n\sz_0^{-1} t)$ for $t\in[0,\sigma_0]$.
A direct calculation leads to
$$  \lim_{n \to\fz} A_\lz(\gz^{(n)} ) = \lim_{n \to\fz} A_\lz(\dz_n) = m_\lz(x,y).$$
We may also assume that
$$A_\lz(\gz^{(n)} )\le m_\lz(x,y)+1, \quad\forall n\ge1.$$

It is easy to see that $\{\gz^{(n)} \}$ is uniform bounded and equicontinuous. Indeed,
 for any $s,t\in[0,\sz_0]$ with $s<t$, one has
$$|\gz^{(n)} (s)-\gz^{(n)} (t)|\le \int_s^t|\dot\gz^{(n)} (\tau)|\,d\tau\le (\int_0^{\sz_0}|\dot\gz^{(n)} (\tau)|^2\,d\tau)^{1/2}|s-t|^{1/2}\le [2(m_\lz(x,y)+1)]^{1/2}|s-t|^{1/2}.$$
 In particular,
$$|\gz^{(n)} (0)-\gz^{(n)} (t)|\le   [2(m_\lz(x,y)+1)]^{1/2}| t|^{1/2}\le [2(m_\lz(x,y)+1)]^{1/2} |\sz_0|^{1/2}.$$

By Arzela-Ascoli Theorem, up to some subsequence, we may assume
that the sequence $\{\gz^{(n)} \}$ converges uniformly to a curve $\gz \in  C^0( [0,\sz_0];\rr^{dN})$ with $\gz(0)=x,\gz(\sz_0)=y$.
{\color{black}  Note that $\gz$ is absolutely continuous.  Indeed,  for any $\ez>0$ and for any family $\{[s_i,t_i]\}_{1\le i\le  k}$ of mutually disjoint intervals such that $\sum_{i=1}^k|s_i-t_i| <\ez$, applying
the Cauchy-Schwarz    inequality,
 we have
 \begin{align*}
   \sum_{i=1}^k|\gz^{(n)} (s_i)-\gz^{(n)} (t_i)| & \le \int_{\cup_{1\le i\le  k} [s_i,t_i]}|\dot\gz^{(n)} (\tau)|\,d\tau\\
    & \le \left[\int_{\cup_{1\le i\le  k} [s_i,t_i]} |\dot\gz^{(n)} (\tau)|^2\,d\tau\right]^{1/2} \left[\sum_{i=1}^k|s_i-t_i|\right]^{1/2} \\
    & \le 2^{1/2} [m_\lz(x,y)+1]^{1/2}\left[\sum_{i=1}^k|s_i-t_i|\right]^{1/2}  \\
    & <2^{1/2} [m_\lz(x,y)+1]^{1/2}\ez^{1/2}.
 \end{align*}}
Thus $\gz \in \mathcal{AC} (x, y;[0,\sz_0];\rr^{dN})$.
Apply Tonelli's Theorem for convex Lagrangians to get
$$ \frac12 \int_0^{\sz_0} |\dot \gz(s)|^2 \, ds  \le \frac12 \liminf_{n \to \fz} \int_0^{\sz_0} |\dot \gz^{(n)} (s)|^2 \, ds $$
and Fatou's Lemma to obtain that
$$   \int_0^{\sz_0} F(\gz) \, ds  \le  \liminf_{n \to \fz} \int_0^{\sz_0} F(\gz^{(n)} ) \, ds. $$
Therefore $A_\lz(\gz) \le m_\lz(x, y)$, which is only possible if the equality holds.
Note that  $A_\lz(\gz)<\fz$ also implies that $\gz \in \mathcal{AC} (x, y;[0,\sz_0];\wz\Omega)$.

Next for any $\gz  \in \mathcal {AC}(x, y; [0,\sz];\wz\Omega)$ such that $A_{\lz}(\gz ) = m_\lz(x, y)$ and for any $0\le s <t\le \sz$, we claim that
$$m_\lz(\gz(s),\gz(t))\le A_\lz(\gz|_{[s,t]}).$$
We show it by contradiction. If $$m_\lz(\gz(s),\gz(t))< A_\lz(\gz|_ {[s,t]}) ,$$
we find $\gz_0  \in \mathcal {AC}(\gz(s), \gz(t); [0,\sz_0];\wz\Omega)$ such that $$A_{\lz}(\gz_0 ) = m_\lz(\gz(s), \gz(t))=m_\lz(\gz(s),\gz(t))< A_\lz(\gz|_ {[s,t]}).$$

The {\color{black}concatenation of $\gz|_{[0,s]}$,$\gz_0$ and $\gz|_{[t,\sz]}$} gives a curve $\eta\in \mathcal {AC}(x, y; [0,\sz_0+\sz-
(t-s) ];\wz\Omega)$ such that
\begin{align*}m_\lz(x,y)\le A_\lz(\eta)&= A_\lz(\gz|_{[0,s]})+A_\lz(\gz_0)+A_\lz(\gz|_{[t,\sz]})\\
&< A_\lz(\gz|_{[0,s]})+A_\lz(\gz|_{[s,t]})+A_\lz(\gz|_{[t,\sz]})\\
&=A_\lz(\gz)=m_\lz(x,y),
\end{align*}
which is a contradiction.
\end{proof}
%
%
%

To prove lemma \ref{key00}, we need the following
 auxiliary Lemma \ref{A5} - Lemma \ref{A8}.
\begin{lemapp} \label{A5}
Given any $\lz>0$ and $x,y\in\Omega$, let $  \gz\in \mathcal {AC}(x,y;[0,\sz];\wz\Omega)$ {\color{black} satisfying} $m_\lz(x,y)=A_\lz(\gz )$.  Then
 $F\circ \gz <\fz$  and $|\dot \gz|>0$ almost everywhere.
\end{lemapp}

{\color{black}
To prove Lemma \ref{A5} and for later use,
we recall the following two change of variable formulas.
We refer to for example \cite[Section 3.3.3, Theorem 2]{eg92}
 for the first one, which  comes from the area formula and works for Lipschitz maps;
and refer to \cite[Proposition 2.2.18]{b97} for  the second one, which  works for absolute continuous maps.

\begin{lemapp}\label{changeofvariable}
(i) Let $f:\rr\to\rr$ be Lipschitz. For any $g\in L^1(\rr)$, one has
$$\left|\sum_{s\in f^{-1}(t)} g (s)\right|\le \sum_{s\in f^{-1}(t)}|g|(s)\in L^1(\rr),$$
where $f^{-1}(t)$ is at most countable  for almost all $t\in\rr$, and
$$\int_\rr g(s)|f'(s)|\,ds=\int_\rr \left[\sum_{s\in f^{-1}(t)}g(s)\right]\,dt. $$
If $f$ is injective in addition, then     $g\circ f^{-1}\in L^1(\rr)$ and
$$\int_\rr g(s)|f'(s)|\,ds=\int_\rr  g\circ f^{-1}(t) \,dt. $$

(ii) Let  $ f:[\alpha,\bz]\to[a,b]$ be  absolutely continuous and increasing,  and satisfy
$a=f(\alpha)$ and $b=f(\bz)$. Then for any
 $g \in L^1([a,b])$,
one has $ (g\circ f)f'\in L^1([\alpha,\bz])$ and
$$\int_a^b g (t)\,dt=\int_{\az}^{\bz}g (f(s))f'(s) \,ds.$$
\end{lemapp}
}

\begin{proof}[Proof of Lemma \ref{A5}]\rm
Note that $A_\lz(\gz)<\fz$ implies the integrability of $F\circ \gz$ in $[0,\sz]$,
and hence    $F\circ \gz(s)<\fz$ for almost all $s\in[0,\sz]$.
Moreover, denote by $E$ {\color{black} the set of $t\in [0,\sz]$ such that $|\dot \gz(t)|=0$}.  We prove $|E|=0$ by contradiction.  Assuming that $|E|>0$ below.

First, we show that {\color{black} there is no interval $[s_0, s_1] \subset  [0, \sz]$ such that
almost all points in $[s_0, s_1]$ are in $E$.}
Otherwise, assume that  $ |\dot\gz |=0  $ in the interval $[s_0,s_1]\subset[0,\sz]$
with $s_0<s_1$.
Then $ \gz(s)= \gz(s_0)$ for $s\in [s_0,s_1]$.
Let $$\mbox{$\eta(s)=\gz(s)$ if $0\le s\le s_0$,
and $\eta(s)=\gz(s+(s_1-s_0))$ if $\sz-(s_1-s_0) \ge s\ge s_0$.}$$
Obviously,   $\eta\in \mathcal {AC}(x,y;[0,{\color{black} \sz}-(s_1-s_0)];\wz\Omega)$ and
$A(\eta)<A(\gz)$, {\color{black} which is a contradiction}.

 Next define
$$\phi(s):=\int^s_0\chi_{[0,\sz]\setminus E}(\dz)\,d\dz.$$
{\color{black} It is obvious that $\phi$ is Lipschitz, and hence absolutely continuous,
  $\phi'(s)=\chi_{[0,\sz]\setminus E}(s)$  for almost all $s\in[0,{\color{black} \sz}]$. Moreover, $|\phi(E)|=0$  and $\phi ([0,\sz])=[0,\sz-|E|]$.
 Since $E$ does not contain any interval, we
 know that $\phi$ is strictly increasing in  $[0,\sz]$, and hence injective.
 Thus the inverse $\phi ^{-1}:[0,\sz-|E|]\to [0,\sz]$ is well-defined.

{\color{black} 
For any $g\in L^1([0,\sz])$,  we claim that
  \begin{align}\label{e.xx1}  \mbox{  $ g \circ \phi^{-1} \in  L^1 ([0,\sz-|E|]) $ and  }
  \int_0^\sz  g\chi_{[0,\sz]\setminus E} \,ds =
  \int_0^{\sz-|E|}g \circ \phi^{-1}\,ds
\end{align}
Indeed,  let $\wz\phi(t)=\int_0^t\chi_{E^\complement} $ for $t\in\rr$.
Then $\wz \phi$ is {\color{black}Lipschitz}  and  strictly increasing in whole $\rr$,
 $\wz \phi|_{[0,\sz]}=\phi$ and $(\wz \phi)^{-1}|_{[0,\sz-|E|]}=\phi^{-1}$.
Applying
 the   change of variable formula in Lemma \ref{changeofvariable}(i) to $g\chi_{[0,\sz]}$ and   $\wz\phi$,  noting  that $\phi[0,\sz]=[0,\sz-|E|]$ and
 $\phi'=\chi_{[0,\sz]\setminus E}$ almost everywhere,
   one has
  \begin{align*}
  \int_0^\sz  g\chi_{[0,\sz]\setminus E} \,ds=
    \int_\rr  (g\chi_{[0,\sz] }) |(\wz \phi)'|\,ds=
      \int_\rr (g\chi_{[0,\sz] })\circ (\wz \phi)^{-1}\,ds=
  \int_0^{\sz-|E|}g \circ \phi^{-1}\,ds
\end{align*}
as desired.

 Write  $\eta(t)=\gamma(\phi^{-1}(t))$ for $t\in[0,\sz-|E|]$.   Since $ \gz$ is absolutely continuous (hence  $\dot\gz\in L^1([0,\sz])$) and $\dot\gz= 0$ in $E$,  for any $t\in[0,\sz-|E|]$ we have
 \begin{align*}
\eta(t )-\eta(0)&=  \gamma(\phi^{-1}(t))-\gamma(0)=
\int_{0}^{\phi^{-1}(t)}\dot\gz(s)\,ds  
= \int_{0}^{\sz } (\dot\gz \chi_{[0,\phi^{-1}(t)]}\chi_{[0,\sz]\setminus E})(s)   \,ds.
 \end{align*}
  Applying \eqref{e.xx1} to $\dot\gz \chi_{[0,\phi^{-1}(t)]}$ for all $t\in[0,\sz-|E|]$,
  one  has $\dot \gz  \circ \phi^{-1} \in  L^1 ([0,\sz-|E|]) $ and
\begin{align*}
\eta(t )-\eta(0)
= \int_{0}^{\sz -|E|} (\dot\gz \chi_{[0,\phi^{-1}(t)]})\circ\phi^{-1}(s)    \,ds.
 \end{align*}
Thanks to $\chi_{[0,\phi^{-1}(t)]}\circ \phi^{-1}=\chi_{[0, t]} $, we obtain
\begin{align*} \eta(t )-\eta(0) =\int_{0}^{\sz-|E|}\dot\gz   \circ\phi^{-1}(\dz) \chi_{[0,t]}  (\dz)  \,d\dz=\int_0^t  \dot\gz \circ\phi^{-1}(\dz)    \,d\dz.
 \end{align*}
 }}
Thus
$\eta \in \mathcal {AC}(x,y;[0,\sz-|E|],\wz\Omega)$
 and hence differentiable almost everywhere in  $[0,\sz-|E|]$ with
\begin{align}\label{e.xx21} \dot \eta =  \dot\gz  \circ\phi^{-1}   \in L^1([0,\sz-|E|]).
\end{align}

 {\color{black}
By  \eqref{e.xx21}
and $ |\dot \gz|^2\in L^1([0,\sz])$,  applying \eqref{e.xx1} to $|\dot \gz|^2$ and $F\circ \gz \in L^1([0,\sz])$ one has
\begin{align*}
  A_\lz(\eta) & =
\int_0^{\sz-|E|}[ \frac12| \dot\eta(t)|^2+F\circ\eta(t) +\lz]\,dt \\
& =
\int_0^{\sz-|E|}[ \frac12|\dot\gz\circ\phi^{-1}  (t)|^2+(F\circ\gamma)\circ \phi^{-1}   (t) +\lz]\,dt \\
& =
\int_0^{\sz }[ \frac12| \dot\gz  (t)|^2 \chi_{[0,\sz]\setminus E} (t)+(F\circ\gamma)   (t)  \chi_{[0,\sz]\setminus E} (t) ]\,dt +
\lz(\sz-|E|)
\end{align*}
Since $|[0,\sz]\setminus E|=\sz-|E|$,  one has
  \begin{align*}
  A_\lz(\eta) & =
\int_0^{\sz }[ \frac12|\dot \gz (s)|^2+F\circ \gz(s) +\lz]\chi_{[0,\sz]\setminus E}(s)\,ds
<A_\lz(\gz).
\end{align*}
}
  Hence $A_\lz(\eta) < m_\lz(x,y)$, which contradicts to the definition of $m_\lz(x,y)$.
\end{proof}

\begin{lemapp}
Let $\gz\in \mathcal {AC}(x,y;[0,\sz];\rr^{dN})$ with  $\wz A_\lz(\gz)<\fz$.
We can reparameterize $\gz$ to get a new curve
$\xi\in  \mathcal {AC}(x,y;[0,\tau];\rr^{dN})$  so that  $|\dot \xi|>0$ almost everywhere and  $\wz A_\lz (\xi)=\wz A_\lz (\gz) $.
\end{lemapp}

\begin{proof}
The proof is much similar to that of {\color{black}Lemma \ref{A5}}.
\end{proof}

\begin{lemapp} \label{A8}
Let $\gz\in \mathcal {AC}(x,y;[0,\sz]; \rr^{dN})$ with $\wz A_\lz(\gz)<\fz$, and
 $|\dot \gz|>0$ almost everywhere. Then $\gz\in \mathcal {AC}(x,y;[0,\sz]; \wz\Omega)$, and
we can reparameterize $\gz$ to get a new curve
$\eta\in  \mathcal {AC}(x,y;[0,\tau];\wz\Omega)$    such that
$$|\dot \eta(t)|=\sqrt{2(F\circ \eta(t) +\lz)}\quad \mbox{almost everywhere.}$$
 Moreover,
$$A_\lz (\eta)=\wz A_\lz (\eta)=\wz A_\lz (\gz)\le A_\lz(\gz).
$$

\end{lemapp}

\begin{proof}
Write
$$\psi(s)=\int_0^s\frac{|\dot\gz(\dz)|}{\sqrt{2 F\circ\gz(\dz)+2\lz}}\,d\dz\quad\forall   s\in[0,\sz].$$
Note that $\psi(0)=0$ and $\psi(\sz)<\fz$.
Obviously, $\psi$ is absolutely continuous.
\begin{equation}\label{e.xx3}\psi'(s)= \frac{|\dot\gz(s)|}{\sqrt{2 F\circ\gz(s)+2\lz}}\quad\mbox{for almost all $s\in[0,\sz]$}.
\end{equation}
Note that {\color{black} $|\dot\gz(s)|\sqrt{2 F\circ\gz(s)+2\lz} \in L^1([0,\sz])$ and $|\dot\gz|>0$ almost everywhere, hence} $F\circ \gz<\fz$ almost everywhere.
Since $|\dot \gz|>0$  almost everywhere, we have  $\psi' >0$ almost everywhere.  Thus    $\psi$ is continuous and strictly increasing.   Therefore $\psi([0,\sz])=[0,\psi(\sz)]$, and
$\psi^{-1}:[0,\psi(\sz)]\to[0,\sz]$ is  also continuous, strictly increasing.

Next we show that  $\psi^{-1}$ is absolutely continuous,
that is, $|\psi^{-1}(E)|=0$  whenever $E \subset[0,\psi(\sz)]$ has measure $|E|=0$.
We only need to prove that  for any set $E\subset[0,\sz]$, if $ \psi(E) $ has measure $0$, then $E$ has measure $0$.
Indeed, $\psi(E)$ must be contained in a $G_\dz$-set $H$ which has measure $0$, and hence $E\subset\psi^{-1}(H)$.
So if
 $\psi^{-1}(H)$ has measure $0$, then $E$ has measure $0$.
Since $H$ is a $G_\dz$-set, we know that $\psi^{-1}(H)$ is also a $G_\dz$-set, one can see that
$$0=|H|=\int_{\psi^{-1}(H)}\psi'(s)\,ds.$$
Since $\psi'>0$ almost everywhere, we have $|\psi^{-1}(H)|=0$ as desired.

{\color{black} Let $\eta(t)=\gz(\psi^{-1}(t))$ for $t\in[0,\psi(\sz)]$.
Thus,  for any $t\in[0,\psi(\sz)]$, one has
$$\eta(t)-\eta(0)=\int_0^{\psi^{-1}(t)}\dot\gz(\dz)\,d\dz=\int_0^\sz (\dot\gz\chi_{[0,\psi^{-1}(t)]}) (\dz)\,d\dz.$$
Applying the  change of variable formula given in Lemma \ref{changeofvariable}(ii) to $\dot\gz\chi_{[0,\psi^{-1}]}\in L^1([0,\sz]) $ and
$\psi^{-1}$,  noting
 $\chi_{[0,\psi^{-1}(t)]}\circ\psi^{-1} =\chi_{[0,t]}$,
one has $\dot\gz \circ \psi^{-1} (\psi^{-1})'\in L^1([0,t])$ and
$$\eta(t)-\eta(0)=\int_0^{\psi(\sz)}(\dot\gz\chi_{[0,\psi^{-1}(t)]})(\psi^{-1}(s))(\psi^{-1})'(s) \,ds=\int_0^t \dot\gz (\psi^{-1}(s))(\psi^{-1})'(s) \,ds.$$
 Thus $\eta \in \mathcal{AC}(x,y;[0,\psi(\sz)];\wz\Omega)$ with
 $\dot\eta =\dot \gz\circ\psi^{-1} (\psi^{-1})' $  almost everywhere.

 Since $t=\psi\circ \psi^{-1}(t)$ for all $t\in [0,\psi(\sz)]$, by the chain rule  we have
 $1=\psi'\circ \psi^{-1} (\psi^{-1})' $ almost everywhere in $[0,\psi(\sz)]$.
Denote by $E$  the set where $\psi$ is differentiable and $\psi'>0$.
Since $|[0,\sz]\setminus  E|=0$, by the absolute continuity of $\psi$,  we have
$|[0,\psi(\sz)]\setminus \psi(E)|=|\psi ([0,\sz]\setminus  E)|=0.$  Thus  $\psi'\circ \psi^{-1}>0$ in $ \psi (  E)$, and hence, almost everywhere
in $[0,\psi(\sz)]$.  Recalling \eqref{e.xx3}, we obtain
  $$(\psi^{-1})'(t) =\frac1{\psi'\circ \psi^{-1}(t)}= \frac{\sqrt{2 F\circ\gz\circ \psi^{-1}(t)+2\lz}}{|\dot\gz\circ \psi^{-1}(t)|}\quad\mbox{for almost all $t\in[0,\psi(\sz)]$},$$
and hence
\begin{equation*}
|\dot \eta(t)|
  ={\sqrt{2 F\circ\gz(\psi^{-1}(t))+2\lz}}={\sqrt{2 F\circ\eta( t )+2\lz}} \quad\mbox{for almost all $t\in[0,\psi(\sz)]$}.
\end{equation*}}
Thus  $$[\frac12|\dot\eta(t)|^2+F\circ\eta(t)+\lz]= |\dot \eta(t)|\sqrt{2 F\circ\eta(t)+2\lz} \quad\mbox{for almost all $t\in[0,\psi(\sz)]$}.$$
We therefore obtain that
$$A_\lz(\eta)=\wz A_\lz(\eta)=\wz A_\lz(\gz)$$
as desired.
\end{proof}

We are now in a position to show

\begin{proof} [Proof of Lemma \ref{key00}.]
Note that for any $\gz \in \mathcal {AC}(x,y;[0,\sz];\rr^{dN})$,
by Cauchy-Schwartz inequality,  one has
$$ \frac 12  |\dot \gz (s)|^2+ F(\gz(s))+\lz
\ge 2\sqrt {F(\gz(s))+\lz} \sqrt{\frac 12  |\dot \gz (s)|^2} =
\sqrt{2(F(\gz(s))+\lz) }     |\dot \gz (s) |,$$
which gives
$A_\lz(\gz)\ge \wz A_\lz(\gz).$  Thus $m_\lz(x,y)\ge \wz m_\lz(x,y).$

We prove by contradiction that $|\dot \gz(t)|=\sqrt{2(F\circ \gz(t) +\lz)}  \,\, \mbox{almost everywhere.}$
 Suppose that this   is not correct.
Write $$E=\{ s\in[0,\tau] : |\dot \gz (s)|\ne \sqrt{2(F(\gz(s))+\lz)}, F(\gz(s))<\fz   \}.$$
Then $|E|>0$.
Moreover for   $s\in E$, one has
$$ \frac 12  |\dot \gz (s)|^2+ F(\gz(s))+\lz
> 2\sqrt {F(\gz(s))+\lz} \sqrt{\frac 12  |\dot \gz (s)|^2} =
\sqrt{2(F(\gz(s))+\lz) }     |\dot \gz (s) |.$$
Thus
$$m_\lz(x,y)=A_\lz(\gz)>\wz A_\lz(\gz).$$
Reparameterize $\gz$ to get $\eta$ as in Lemma A.5 we know that
$$A_\lz(\eta)=\wz A_\lz(\eta) =\wz A_\lz(\gz).$$
Note that $m_\lz(x,y)\le A_\lz(\eta)$. This is a contradiction.

It is easy to see that  $m_\lz(x,y)\ge \wz m_\lz(x,y).$
We prove  $ \wz m_\lz(x,y)= m_\lz(x,y)  $ by contradiction. Suppose this {\color{black}is not correct}.
Then  $ \wz m_\lz(x,y)< m_\lz(x,y)  $.
 There exists a curve {\color{black}$\gz$} such that
 $$\wz m_\lz(x,y)\le
\wz A_\lz(\gz)< m_\lz(x,y).$$
{\color{black} Reparameterizing $\gz$ like in Lemma \ref{A5}, we find a curve $\eta$ such that}
$$A_\lz(\eta)=\wz A_\lz(\eta)=\wz A_\lz(\gz).$$
Ssince  $m_\lz(x,y)\le A_\lz(\eta)$, this is a contradiction.
\end{proof}

\begin{proof} [Proof of Lemma \ref{A6}] We only need to prove  Lemma \ref{A6} for
 $m_\lz$-geodesics    with canonical parameter.
 Assume that $\gz\in \mathcal(x,y;[0,\sz];\wz\Omega)$ is an
 $m_\lz$-geodesic    with canonical parameter joining  $x,y$, and moreover   $\gz|_{(0,\sz)}\subset\Omega $.
Up to considering $\gz_{[\dz,\sz-\dz]} \subset\Omega$ for any sufficiently small $\dz>0$,
we may assume that $\gz\subset\Omega$. {\color{black} Recall that
  $$ m_\lz(x,y)=A_\lz(\gz)\le  A_\lz(\gz+\ez\xi),\quad\forall \xi\in  C^1(0,0;[0,\sz];\rr^{dN}).$$
 Since $\gz\subset\Omega$, there exist $ \ez_0$ depending on $\gz$ and $\xi$ such that
$ A_\lz(\gz+\ez\xi)<\fz$  for any $\ez<\ez_0$.
Thus  $$0=\frac{d}{d\ez}|_{\ez=0} A_\lz(\gz+\ez\xi),$$
that is,
$$
0=\frac{d}{d\ez}|_{\ez=0} \int_0^\sz    ( \frac12 |(\dot\gz +\ez \dot\xi)(s)|^2+ F\circ(\gz +\ez\xi)(s) +\lz)  \,ds.$$}
A direct calculation gives
{\color{black}
\begin{equation*}
\frac{d}{d\ez}|_{\ez=0} [ \frac12 |(\dot\gz +\ez \dot\xi)(s)|^2+ F\circ(\gz +\ez\xi)(s) +\lz]=
 \langle \langle \dot\gz  (s), \dot\xi(s) \rangle \rangle +
 \langle \langle \nabla F(\gz (s)), \xi(s) \rangle \rangle.
\end{equation*}}
Therefore we {\color{black}obtain}
 $$\int_0^\sz  {\color{black}[ \langle \langle \dot\gz (s), \dot\xi(s) \rangle \rangle + \langle \langle \nabla F(\gz (s)), \xi(s)\rangle \rangle] }\,ds=0, \quad\forall \xi\in  C^1(0,0,[0,\sz];\rr^{dN}).$$
Note that $\gz\subset\Omega$ is free of collision. { Since $F \in C^2(\Omega)$, it follows from \cite[Chapter 3]{fa} that $\gz \in C^2$ }, and hence,
 $$\int_0^\tau {\color{black} \langle \langle [-\ddot\gz (s)+ \nabla F(\gz (s)) ], \xi(s)\rangle \rangle}\,ds=0.$$
By the {\color{black}arbitrariness} of $\xi$,
we have
 $\ddot\gz (s)=\nabla F(\gz (s)) $
for all $s \in (0,\sz)$ as desired.
\end{proof}

{\color{black}
Finally, we prove   \eqref{ea.x1} in the Step 3 in Remark \ref{eg1}(iii).

\begin{proof}[Proof of   \eqref{ea.x1}]

Let all notations and notions be as in Remark \ref{eg1}(iii).
In particular, recall   $$\mbox{$\Sigma^-_{12} =\{x=(x_1,\cdots,x_N)\in\rr^{dN}: x_1=x_2\in\Gamma\}$ and
$\Gamma= \{x_i \in \rr^{d } :  x_i^j < 0 , \ \forall 1 \le j\le d \}$}.$$
Let $x\in \Sigma_{12}^-$ be an arbitrary point. We have $x_2=x_1\in\Gamma$, that is,
$x_2^j=x_1^j< 0$ for all $1\le j\le d$.
 {\color{black}Set
$$c_\ast(x):= \min\{1,|x_1^1|, \cdots,|x_1^d| \}=\min\{1,-x_1^1 , \cdots,-x_1^d \}.$$}
Then  $c_\ast(x)>0$. To get \eqref{ea.x1},
  that is,  $\inf_{z\in\Omega} m_\lz(z,x)>0$,
obviously it  suffices to prove
\begin{align}\label{ey.a1}  \inf_{z\in\Omega} m_\lz(z,x)   \ge\sqrt{2\lz} \frac12c_\ast(x). 
\end{align}
To see \eqref{ey.a1},  given any $z\in \Omega$, let
 $\gz\in \mathcal {AC}(z,x;[0,\sz];\wz\Omega)$   for   some $\sz>0$    be an  $m_\lz$-geodesic      with canonical parameter  so that
$A_\lz(\gz)=m_\lz(z,x)<\fz$; see  Lemma A.2 for its existence.
From  Lemma A.2 and \eqref{ew.1}   it follows that
\begin{equation}\label{e.xx9-1}m_\lz(z,x)= A_\lz(\gz)  \ge  A_\lz(\gz|_{[s,\sz]})= m_\lz(\gz(s),x)
 \ge \sqrt{ 2\lz} |\gz(s)-x|\quad\forall s\in[0,\sz). \end{equation}
We then claim that
\begin{equation}\label{e.xx9}
 \mbox{there exists an $s\in[0,\sz)$ such that}\ |\gz(s)-x|\ge \frac12c_\ast(x).
 \end{equation}
Obviously,  from \eqref{e.xx9-1} and the claim \eqref{e.xx9} it follows that
 $m_\lz(z,x) \ge \frac12\sqrt{ 2\lz}c_\ast(x).$ Thus  \eqref{ey.a1} holds  as desired.

Below we prove the above claim \eqref{e.xx9} by contradiction.
Assume that   the claim \eqref{e.xx9} is not correct, in other words, assume that
\begin{equation}\label{e.xx9a} |\gz(s)-x|< \frac12 c_\ast(x)\quad \mbox{ for all $ s\in[0,\sz]$.} \end{equation}
Write $\gz=(\gz_1,\cdots,\gz_N)$ and $\gz_i=(\gz^1_i,\cdots,\gz_i^d)$.
Given any $i=1,2$ and  $1\le j\le d$,
by the hypothesis \eqref{e.xx9a} one has  $$|\gz_i^j(s)- x_i^j|\le |\gz(s)-x|< \frac12 c_\ast(x)
\quad\mbox{ and hence }\quad  \gz_i^j(s)\le x_i^j+ \frac12 c_\ast(x).
$$
Noting $0<c_\ast(x)\le -x_i^j$,
we get $$ \gz_i^j(s)\le \frac12 x_i^j<0. $$
Thus,
 $(\gz_1(s),\gz_2(s))\in \Gamma\times \Gamma$ and hence, by the definition of  $F$, we have
$$F\circ \gz= h_-\circ(|\gz_1-\gz_2|).$$
Define an auxiliary function
   $$\mbox{$v(t)=|\gz_1(t)-\gz_2(t)|$ for all $t\in[0,\sz]$}.$$
Then  $F\circ \gz= h_-\circ v.$
{\color{black}
Moreover, note that $v\in C^0([0,\sz])$. By $x_1=x_2$,
 $$v(\sz)=|\gz_1(\sz)-\gz_2(\sz)|=|x_1-x_2|=0;$$
 while by $z\in\Omega$, we have $z_1\ne z_2$
 and hence  $$v(0)=|\gz_1(0)-\gz_2(0)|=|z_1-z_2|>0.$$
Set $$\kz:=\min\{s\in [0,\sz] \ | \ v(s)=0\}.$$
Then $0<\kz\le \sz$.  Below we  will  prove that
\begin{equation}\label{contra}\int_0^{\kz}
|\dot \gz |\sqrt{ h_-\circ v }\,ds=\fz.
\end{equation}
Before  we give the proof of \eqref{contra}, we show that
\eqref{contra} leads to a contradiction.   Indeed,
  Lemma A.2 and Lemma A.3 yield
 $$m_\lz(z,x)\ge m_\lz(z,\gz(\kz) )\ge\wz A_\lz(\gz|_{[0,\kz]}).$$
Write
$$
\wz A_\lz(\gz|_{[0,\kz]})=  \int_0^{\kz} |\dot \gz |\sqrt{2 F\circ \gz+2 \lz }\,ds
=\int_0^{\kz} |\dot \gz |\sqrt{2 h_-\circ v+2 \lz }\,ds\ge \sqrt2\int_0^{\kz}
|\dot \gz |\sqrt{ h_-\circ v  }\,ds.
$$
We see that
\eqref{contra} implies  $\wz A_\lz(\gz|_{[0,\kz]})=\fz$ and hence
 $m_\lz(z,x)=\fz$. However, since
    $z\in\Omega$, we have  $m_\lz(z,x)<\fz$.
 This is a contradiction.
 Therefore, we conclude that the hypothesis
 \eqref{e.xx9a} is not correct, and hence,  the claim \eqref{e.xx9} must hold as desired.}

Finally, we prove
  \eqref{contra} via the following 4 substeps.

\medskip
{\it Substep 1.}
From the definition of $\kz$ it follows that
  $$\mbox{$v(\kz)=0$,  $v(t)>0$ for $t\in[0,\kz)$, $\gz(\kz)\in  \Sigma^{-}_{12}$ and
$\gz|_{[0,\kz)}\subset\rr^{dN}\setminus \Sigma^{-}_{12}$. }$$
Set $$\tau=\sup_{s\in[0,\kz]} v(s).$$
The choice of $\kz$ and the continuity of $v$   give that $v([0,\kz])=[0,\tau]$.
Moreover, we observe that  $\tau\le c_\ast(x)\le1$. Indeed,
by $x_1=x_2$ and the triangle inequality, for any $s\in[0,\kz]$  we have
$$v(s)=|\gz_1 (s)-\gz_2 (s)|\le  |\gz_1 (s)-x_1 (s)| +|x_2 (s)-\gz_2 (s)| \le 2|\gz-x|  $$
and hence, by the hypothesis
\eqref{e.xx9a}, $v(s)\le c_\ast(x)$.

For any $s\in[0,\kz]$, recall that
 \begin{equation*}
  h_-\circ v(s) =\left\{
\begin{split}
&1 \quad& \mbox{if }    v(s)\in \Lambda=\{0\}\cup(\cup_{k\ge3}  [2^{-k^2-1}, 2^{-k^2+1}]); \\
&\frac1{v(s)^2}\ & \mbox{if} \quad  v(s)\in [0,\tau]\setminus \Lambda.
\end{split}\right.
\end{equation*}
Wrtie $E:=\{s\in[0,\kz]|v(s)\in [0,\tau]\setminus\Lambda\}$.Then
 $$h_-(v(s))=\frac1{v(s)^2}\quad \mbox{ for $s\in E$}$$ and
 $$\chi_E(s) = \chi_{[0,\tau]\setminus \Lambda}(v(s))= \chi_{[0,\tau]\setminus \Lambda}\circ v(s)\quad\forall s\in[0,\kz].$$
 Thus
 \begin{equation} \label{contra-2}
\int_0^{\kz}
|\dot \gz |\sqrt{ h_-\circ v  }\,ds\ge   \int_0^\kz \chi_E(s)
|\dot \gz(s) |\frac1{v(s)}\,ds= \int_0^\kz  \chi_{[0,\tau]\setminus \Lambda}\circ v(s)
|\dot \gz (s)| \frac1{v(s)}\,ds.
\end{equation}

\medskip
{\it Substep 2.}  Since both of $\gz_1$ and $\gz_2$ are absolutely continuous in $[0,\kz]$, we know that $v$ is also absolutely continuous in $[0,\kz]$, and hence $\dot v\in L^1([0,\kz])$.
One further has
\begin{equation}\label{e.xx13a}
|\dot v(t)|= \frac{|\langle (\dot \gz_1-\dot\gz_2)(t),   \gz_1(t)- \gz_2 (t)\rangle |}{v(t)}\le
|(\dot \gz_1-\dot\gz_2)(t)|\le 2
  |\dot \gz(t)|   \quad \mbox{  for almost all $t\in [0,\kz]$.}
 \end{equation}
Applying this in \eqref{contra-2},  we obtain
\begin{equation}\label{contra-1}\int_0^{\kz}
|\dot \gz |\sqrt{ h_-\circ v  }\,ds\ge \frac12\int_0^\kz  \chi_{[0,\tau]\setminus \Lambda}\circ v(s)
|\dot v (s)| \frac1{v(s)}\,ds.
\end{equation}

\medskip
{\it Substep 3.} We claim that
 \begin{align}\label{e.xx10}
 \int_{0}^{\kz}  \frac1{v }|\dot v | \chi_{[0,\tau]\setminus\Lambda}\circ v \,dt&\ge   \int_{[0,\tau]\setminus {\Lambda}} \frac1{t}\,dt,
\end{align}
whose proof will be given in Substep 4.
Assume  that the claim \eqref {e.xx10} holds for the moment.  By  \eqref{contra-1}
and the claim \eqref {e.xx10},  we obtain
$$\int_0^{\kz}
|\dot \gz |\sqrt{ h_-\circ v  }\,ds>   \int_{[0,\tau]\setminus {\Lambda}} \frac1{t}\,dt.
$$
 Let $k_\tau \ge 3$ such that $2^{-k_\tau^2+1}\le \tau$.  Then
$$ \int_{[0,\tau]\setminus {\Lambda}} \frac1{t}\,dt\ge  \sum_{k\ge k_\tau} \int_{2^{-k^2+1}}^{2^{-(k-1)^2-1}}\frac1{t}\,dt=   \sum_{k\ge k_\tau} \ln \frac{2^{-(k-1)^2-1}}{2^{-k^2+1}}
\ge   \sum_{k\ge k_\tau} \ln  2^{2k-3}  \ge
 (\ln2 )\sum_{k\ge k_\tau} k=\fz.
  $$
     Thus $$ \int_0^{\kz}
|\dot \gz |\sqrt{ h_-\circ v  }\,ds=\fz,$$
 which gives \eqref{contra}.

\medskip
{\it Substep 4.} We prove the claim \eqref{e.xx10} via Lemma \ref{changeofvariable}(i).
Since $v$ is only  defined in $[0,\sz]$ and also {\color{black} not necessarily Lipschitz} in $ [0,\kz]$, we can not use
Lemma \ref{changeofvariable}(i)  directly.
{\color{black}
To overcome this difficulty, we show that the restriction of $v$ in  subintervals  $[0,{\kz}-\ez]$ is Lipschitz for all sufficiently small $\ez>0$, and extend them to $\rr$ via the McShane's extension.
To be precise, by  Lemma A.3 and $F\circ \gz=h_-\circ v$, one has
$$ |\dot \gz |= \sqrt{2(F\circ \gz  +\lz)} =  \sqrt{2(h_-\circ v+\lz)}
 \quad \text{   almost  everywhere in } [0,\kz] .$$
 Recall that $0<v  \le  1$ in $[0,\kz)$, $\lz>0$, and $h_-(t)\le\frac1{t^2}$ by definition, we have
 \begin{equation*}
 h_-\circ v+\lz  \le  (1+\lz) \frac1{v^2}\quad \mbox{     in $[0,\kz]$.}
 \end{equation*}
Thus, by  \eqref{e.xx13a},
 $$|\dot v | \le  2|\dot\gz|\le 2\sqrt{2(1+\lz) }\frac1{v } \quad \mbox{  almost everywhere   in $[0, \kz]$.}$$
   For $0<\ez<{\kz} $, set
$$\dz_\ez:=\min\{v(t),t\in[0,{\kz} -\ez]\}.$$
Since $v$ is continuous in $[0,{\kz} ]$ and $v>0$ in $[0,\kz)$,
we know that $\dz_\ez>0$.  Thus
  $$|\dot v|\le 2\sqrt{2 (1+\lz) } \frac1{\dz_\ez }
 \mbox{\ almost everywhere in $[0,\kz-\ez]$.}
$$
From this and the absolute continuity of $v$ in $[0,\kz]$,
we know that $v$ is Lipschitz in $[0,\kz-\ez]$ with
$$|v(s)-v(t)|=|\int_s^t\dot v(\xi)\,d\xi|\le 2\sqrt{2 (1+\lz) } \frac1{\dz_\ez }|s-t|,\quad \forall s,t\in[0,\kz-\ez].$$
  Denote by $\wz {v^\ez}$ the McShane's extension of $v|_{[0,\kz-\ez]}$ into $\rr$, that is,
\begin{equation}\label{McSh}\wz {v^\ez}(t)=\inf \left\{v (s)+2\sqrt{2 (1+\lz) } \frac1{\dz_\ez }|t-s| \ \bigg| \ s\in[0,\kz-\ez]
\right\}, \quad\forall t\in\rr.
\end{equation}
Then
$$|\wz {v^\ez}(s)-\wz {v^\ez}(t)|\le 2\sqrt{2 (1+\lz) } \frac1{\dz_\ez }|s-t|,\quad\forall s,t\in\rr,$$
that is, $\wz {v^\ez}$ is Lipschitz in $\rr$, and moreover,
$\wz {v^\ez}|_{[0,\kz-\ez]}=v|_{[0,\kz-\ez]}$;
see for example [23,Chapter 6].

Moreover,
write \begin{align*}\int_{0}^{\kz}   \frac1{v }|\dot v | \chi_{[0,\tau]\setminus\Lambda}\circ v \,dt=  \lim_{\ez\to0} \int_{0}^{{\kz} -\ez} \frac1{v }|\dot v |\chi_{[0,\tau]\setminus {\Lambda}}\circ v \,ds.
\end{align*}
Since $v$ is continuous in $[0,{\kz} ]$, $v>0$ in $[0,\sz)$ and $\vz({\kz} )=0$,
we know that   $\dz_\ez\to0$ as $\ez\to0$.
When $\ez>0$ is sufficiently small, one also has
$v([0,\kz-\ez])=[\dz_\ez,\tau]$, that is,
$$\chi_{[0,\tau]\setminus {\Lambda}}\circ v(s)= \chi_{[\dz_\ez,\tau]\setminus {\Lambda}}\circ v
(s)\quad\quad \mbox{ for any } s\in[0,{\kz} -\ez].$$
Therefore, for all sufficiently small $\ez>0$, one has
\begin{align*}
 \int_{0}^{{\kz} -\ez} \frac1{v }|\dot v |\chi_{[0,\tau]\setminus {\Lambda}}\circ v \,ds=\int_{0}^{{\kz} -\ez} \frac1{v }|\dot v |\chi_{[\dz_\ez,\tau]\setminus {\Lambda}}\circ v \,ds.
\end{align*}
Denote by $\wz v^\ez$   the McShane's extension of $v|_{[0,\kz-\ez]}$ into $\rr$ as in
\eqref{McSh}.
Since $\wz {v^\ez}|_{[0,\kz-\ez]}=v|_{[0,\kz-\ez]}$,
   it follows that
\begin{align*}
 \int_{0}^{{\kz} -\ez} \frac1{v }|\dot v |\chi_{[0,\tau]\setminus {\Lambda}}\circ v \,ds=  \int_\rr \chi_{[0,\kz-\ez]}(s)\chi_{[\dz_\ez,\tau]\setminus {\Lambda}}\circ \wz {v^\ez} (s) \frac1{ \wz {v^\ez} (s)}|(\wz  {v^\ez})'(s)|\,ds.
\end{align*}
Since  $\wz {v^\ez}$ is Lipschitz in $\rr$ and
 $ \chi_{[0,\kz-\ez]}\frac{1}{ \wz {v^\ez} }\chi_{[\dz_\ez,\tau]\setminus {\Lambda}}\circ  \wz {v^\ez}  \in L^1(\rr)$,
we are able to apply  the change of variable formula in Lemma \ref{changeofvariable}(i) with   $f=\wz {v^\ez}$ in $\rr$  and   $g=\chi_{[0,\kz-\ez]}\frac{1}{ \wz {v^\ez} }\chi_{[\dz_\ez,\tau]\setminus {\Lambda}}\circ  \wz {v^\ez}$ therein to get
\begin{align}\label{A9}\int_{0}^{{\kz} -\ez} \frac1{v }|\dot v |\chi_{[0,\tau]\setminus {\Lambda}}\circ v \,ds=\int_{\rr}\sum_{s\in (\wz  {v^\ez})^{-1}(\{ t\})} \left[\chi_{[0,\kz-\ez]}(s)\chi_{[\dz_\ez,\tau]\setminus {\Lambda}}\circ \wz {v^\ez} (s) \frac1{\wz {v^\ez}(s)}\right] \,dt.
\end{align} }
For any sufficiently small $\ez>0$,
 since $v([0,\kz-\ez])=[\dz_\ez,\tau]$,  given any $ t\in [\dz_\ez,\tau]\setminus \Lambda$,  one can find at least one $s\in[0,\kz-\ez]$
such that $\wz v^\ez(s)= v(s)=t$, and hence,
{\color{black}
$$\sum_{s\in (\wz  {v^\ez})^{-1}(\{ t\})} \left[\chi_{[0,\kz-\ez]}(s)\chi_{[\dz_\ez,\tau]\setminus {\Lambda}}\circ \wz {v^\ez}(s) \frac1{ \wz {v^\ez}(s)}\right]\ge \frac1t.$$}
From this and \eqref{A9} one deduces that
$$
\int_{0}^{{\kz} -\ez} \frac1{v }|\dot v |\chi_{[0,\tau]\setminus {\Lambda}}\circ v \,ds
\ge\int_{[\dz_\ez,\tau]\setminus\Lambda}\frac1t\,dt.$$
Sending
$\ez\to0$, by $\dz_\ez\to0$   one gets
\begin{align*}\int_{0}^{\kz}   \frac1{v }|\dot v | \chi_{[0,\tau]\setminus\Lambda}\circ v \,dt
&
\ge     \int_{[0,\tau]\setminus {\Lambda}} \frac1{t}\,dt,
\end{align*}
which gives the claim \eqref{e.xx10}
as desired. The proof is complete.
\end{proof}

\medskip

\noindent {\bf Data Availability Statement.} All data, models, and code generated or used during the study appear in the submitted article.

\noindent Jiayin Liu

\noindent
School of Mathematical Science, Beihang University, Changping District Shahe Higher Education Park
  South Third Street No. 9, Beijing 102206, P. R. China

{\it and}

\noindent Department of Mathematics and Statistics, University of Jyv\"{a}skyl\"{a}, Jyv\"{a}skyl\"{a},  Finland

\noindent{\it E-mail }:  \texttt{jiayin.mat.liu@jyu.fi}

\bigskip
\noindent Duokui Yan

\noindent School of Mathematical Science, Beihang University, Changping District Shahe Higher Education Park
  South Third Street No. 9, Beijing 102206, P. R. China

\noindent {\it E-mail }:
\texttt{duokuiyan@buaa.edu.cn}

\bigskip
\noindent Yuan Zhou

\noindent School of Mathematical Sciences, Beijing Normal University, Haidian District Xinejikou Waidajie No.19, Beijing 100875, P. R. China

\noindent {\it E-mail }:
\texttt{yuan.zhou@bnu.edu.cn}

\end{document}